\documentclass{amsart}

\usepackage{float}
\usepackage[colorlinks=True, citecolor=blue, linkcolor=blue]{hyperref}
\usepackage{comment}
\usepackage{amssymb}
\usepackage{tikz-cd}
\usepackage{mathtools}
\usepackage[nameinlink]{cleveref}
\usepackage{booktabs}
\usepackage{multirow}

\theoremstyle{definition}
\newtheorem{definition}{Definition}[section]
\newtheorem{example}[definition]{Example}
\newtheorem{remark}[definition]{Remark}

\theoremstyle{plain}
\newtheorem{theorem}[definition]{Theorem}
\newtheorem{lemma}[definition]{Lemma}
\newtheorem{proposition}[definition]{Proposition}
\newtheorem{corollary}[definition]{Corollary}

\DeclareMathOperator{\spin}{spin}
\DeclareMathOperator{\Tr}{Tr}
\DeclareMathOperator{\Cl}{C}
\DeclareMathOperator{\cl}{cl}
\DeclareMathOperator{\rk}{rk}

\DeclareMathOperator{\id}{id}

\DeclareMathOperator{\Hom}{Hom}
\DeclareMathOperator{\NS}{NS}
\DeclareMathOperator{\Aut}{Aut}

\DeclareMathOperator{\scale}{\mathfrak{s}}
\DeclareMathOperator{\norm}{\mathfrak{n}}
\DeclareMathOperator{\aut}{Aut}
\DeclareMathOperator{\Mo}{Mon}
\DeclareMathOperator{\ns}{NS}
\DeclareMathOperator{\im}{im}

\renewcommand{\AA}{\mathbb{A}}
\newcommand{\QQ}{\mathbb{Q}}
\newcommand{\FF}{\mathbb{F}}
\newcommand{\RR}{\mathbb{R}}
\newcommand{\CC}{\mathbb{C}}
\newcommand{\ZZ}{\mathbb{Z}}
\newcommand{\PP}{\mathbb{P}}
\newcommand{\UU}{\mathbb{U}}
\newcommand{\EE}{\mathbb{E}}
\newcommand{\IZ}{\mathbb{Z}}
\newcommand{\NN}{\mathbb{N}}
\renewcommand{\O}{\mathcal{O}}
\newcommand{\X}{\mathcal{X}}
\newcommand{\T}{\mathcal{T}}
\newcommand{\K}{\mathcal{K}}
\newcommand{\M}{\mathcal{M}}
\newcommand{\F}{\mathcal{F}}
\newcommand{\E}{\mathcal{E}}
\newcommand{\C}{\mathcal{C}}
\newcommand{\B}{\mathcal{B}}
\renewcommand{\L}{\mathcal{L}}

\renewcommand{\P}{\mathcal{P}}

\newcommand{\disc}[1]{#1^{\vee} \! / #1 }
\newcommand{\RN}[1]{\textup{\uppercase\expandafter{\romannumeral#1}}}
\newcommand{\even}{\RN{2}}
\newcommand{\odd}{\RN{1}}

\newcommand{\leg}[2]{\left( \frac{#1}{#2}\right)}
\newcommand{\kumn}{\mathrm{Kum}_n}
\newcommand{\hskn}{K3^{\left[n\right]}}

\makeatletter
\newcommand*{\defeq}{\mathrel{\rlap{%
                     \raisebox{0.3ex}{$\m@th\cdot$}}%
                     \raisebox{-0.3ex}{$\m@th\cdot$}}%
                     =}
\makeatother

\renewcommand{\div}{\mathop{div}}

\newcommand{\condAst}{\hyperlink{ast}{(\ast)}}

\makeatletter
\def\blfootnote{\xdef\@thefnmark{}\@footnotetext}

\title[]{Prime order isometries of unimodular lattices and automorphisms of ihs manifolds}

\author{Simon Brandhorst, Alberto Cattaneo}

\address{Simon Brandhorst,
Fakult\"at f\"ur Mathematik und Informatik, Universit\"at des Saarlandes, Campus E2.4, 66123 Saarbr\"ucken, Germany}
\email{brandhorst@math.uni-sb.de}

\address{Alberto Cattaneo, 
Mathematisches Institut and Hausdorff Center for Mathematics, Universit\"at Bonn, Endenicher Allee 60, 53115 Bonn, Germany; \newline Max Planck Institute for Mathematics, Vivatsgasse 7, 53111 Bonn, Germany.}
\email{cattaneo@math.uni-bonn.de}

\begin{document}
\begin{abstract}
We characterize conjugacy classes of isometries of odd prime order in unimodular $\ZZ$-lattices.
This is applied to give a complete classification of odd prime order non-symplectic automorphisms of irreducible
holomorphic symplectic manifolds up to deformation and birational conjugacy.
\end{abstract}

\maketitle


\tableofcontents

\addtocontents{toc}{\protect\setcounter{tocdepth}{1}}
\section{Introduction}
\subsection*{Isometries of unimodular lattices}
A \emph{lattice} consists of a finitely generated free $\ZZ$-module $L$ equipped with a nondegenerate symmetric bilinear form $b\colon L \times L \rightarrow \QQ$.
It is called \emph{integral} if $b$ is integer valued. An integral lattice is \emph{odd} if there is $x\in L$ with $b(x,x)=1$, otherwise it is called \emph{even}. An integral $\ZZ$-lattice is called \emph{unimodular} if it is of determinant $\pm 1$.
An even unimodular lattice of signature $(l_+, l_-)$ exists if and only if $l_+ \equiv l_- \pmod{8}$, while in the odd case there is no restriction on the signature. If $l_+$ and $l_-$ are both nonzero, then an odd (even) lattice of this signature is unique up to isometry.

We treat the following problem:
does $L$ admit an isometry $f \in O(L)$ of odd prime order $p$? If yes, classify their conjugacy classes.

In order to do so, we define the invariant and coinvariant lattices of $f$ as
\[L^f=\{x \in L \mid f(x)=x\} \quad \mbox{ and }\quad L_f=\left(L^f\right)^\perp.\]
If $L$ is unimodular, these lattices are in fact $p$-elementary. In particular, $|\det L_f|$ is a power of $p$.
Our first main result is the following theorem.
\begin{theorem} \label{thm: classific unimod}
Let $p$ be an odd prime number. There exists a unimodular lattice
$L$ of signature $(l_+,l_-)$ admitting an isometry $f \in O(L)$
of order $p$ with coinvariant lattice $L_f$ of signature $(s_+,s_-)$
and determinant $\det L_f = (-1)^{s_-}p^n$
if and only if there exists
$m\in \ZZ_{\geq 0}$ such that
\begin{enumerate}
\setlength{\itemsep}{5pt}
\item[(I)] $l_+ + l_- - s_+ - s_- > 0$ and $L$ is odd or
\item[(II)] $l_+ \equiv l_- \pmod{8}$ and is $L$ even;
\item $s_+ + s_- = (n+2m)(p-1)>0$;
\item $s_+, s_- \in 2\ZZ$;
\item $s_+ \leq l_+$, $ \quad s_- \leq l_-$;
\item $s_+ + s_- + n \leq l_+ + l_-$;\item if $n = 0$ or $n=l_+ +l_- - s_+ - s_-$,
then $s_+ \equiv s_- \pmod{8}$.
\end{enumerate}
\end{theorem}
The invariants of the theorem determine the genus of the invariant and coinvariant lattice of $f$. But in order to distinguish between conjugacy classes, we need additional invariants.

Let $\zeta_p$ be a primitive $p$-th root of unity.
The coinvariant lattice $L_f$ has the structure of a
$\ZZ[\zeta_p]$-module via $\zeta_p \cdot x = f(x)$.
Moreover, it carries a nondegenerate $\QQ[\zeta_p]$-valued hermitian form $h$ defined by
\[h(x,y) = \sum_{i=0}^{p-1} b(x,f^i(y))\zeta_p^i, \qquad x,y \in L_f.\]

Let $k = \rk_{\ZZ[\zeta_p]} L_f=n+2m$. Its \emph{determinant lattice} $\det(L_f,h_f)$ consists of the top exterior power $\det L_f = \bigwedge^k L_f$ equipped with the hermitian form $\det h_f=\wedge^k h_f$ defined by
$h_f(x,x)=\det h_f(x_i,x_j)_{i,j}$ for $x=x_1 \wedge \ldots \wedge x_k \in \det L_f$.
Let $E=\QQ[\zeta_p]$ and $K= \QQ[\zeta_p + \zeta_p^{-1}]$. Since $\det L_f$ is a module of rank one, it is isomorphic to a fractional ideal $I_f$ of $E$. Seen as an element of the class group $\Cl(E)$ it is independent of the isomorphism and called the \emph{Steinitz class} of $L_f$. It measures the deviation of $L_f$ from being a free $\ZZ[\zeta_p]$-module. The Steinitz class $[I_f]$ is contained in the \emph{relative class group} $\Cl(E/K)$, which is defined as the kernel of the norm map $\Cl(E) \rightarrow \Cl(K)$. The order of $\Cl(E/K)$ is called the \emph{relative class number}.

Finally, the \emph{signatures} of $(L,f)$ are given by the signatures $(k_i^+,k_i^-)\in (2\NN)^2$ of the real quadratic spaces
\begin{equation}\label{eqn:sign}
K_i=\ker(f_\RR+f_\RR^{-1}- \zeta_p^i - \zeta_p^{-i}), \quad i \in \{1,\dots, (p-1)/2\}.\end{equation}
They satisfy $(s_+,s_-)=\sum_{i=1}^{(p-1)/2}(k_i^+,k_i^-)$ and $k_i^+ + k_i^-$ is independent of $i$. In fact any collection of even nonnegative $k_i^\pm$ satisfying these conditions occurs.
\begin{theorem}\label{thm:unimodular_conjugacy}
Let $L$ be a unimodular lattice, $f,g \in O(L)$ be isometries of odd prime order $p$,
$E=\QQ[\zeta_p]$ and $K= \QQ[\zeta_p + \zeta_p^{-1}]$.
Suppose that $L_f$ is indefinite or of rank $p-1$.
Then $f$ is conjugate to $g$ if and only if
\begin{enumerate}
 \item the invariant lattices $L^f$ and $L^g$ are isometric,
 \item $f$ and $g$ have the same signatures,
 \item the determinant lattices $\det(L_f,h_f)$ and $\det(L_g,h_g)$ are isometric.
\end{enumerate}
The number of conjugacy classes with the same signature and invariant lattice as $f$ is given by the relative class number $\# \Cl(E/K)$.

If moreover the relative class number is odd, then the determinant lattices are isometric if and only if $L_f$ and $L_g$ have the same Steinitz invariant in $\Cl(E/K)$.
\end{theorem}

\begin{remark}
For $L_f$ or $L^f$ definite a classification of conjugacy classes includes an enumeration of isometry classes in the respective genera. This is typically of algorithmic nature and has been carried out for instance in \cite{nebe:automorphisms-unimodular,nebe-kirschmer:binary-hermitian} where Kirschmer and Nebe classify extremal unimodular lattices of rank $48$ admitting certain prime order isometries.

The analogous results for $p=2$ are easily derived from \cite[Thms.\ 3.6.2, 3.6.3, \S 16]{nikulin} due to Nikulin. They boil down to a classification of primitive $2$-elementary sublattices of unimodular lattices up to the action of the orthogonal group.

In \cite{quebbemann} Quebbemann classifies odd prime order automorphisms of unimodular lattices taking $(L_f,f)$ and $L^f$ as given. The addition of a classification for hermitian and $p$-elementary lattices allows us to reach the effective results given above.

In a recent breakthrough, Bayer-Fluckiger \cite{bayer-fluckiger:unimodular-charpoly} classifies characteristic polynomials of isometries of unimodular lattices, under the condition that the polynomial has no linear factor, which prevents us from applying her results to our study.
\end{remark}

\subsection*{Automorphisms of ihs manifolds}
Unimodular lattices arise in geometry as the intersection form of closed, oriented $4k$-folds. In the case of simply connected fourfolds, by a celebrated result of Freedman \cite[Thm.\ 1.5]{freedman} the isometry class of the intersection form determines the topology of the manifold (together with a $\ZZ/2\ZZ$ obstruction in the case of odd intersection form). The automorphisms of these manifolds act as isometries on the intersection lattice. \Cref{thm: classific unimod} thus yields restrictions for this action.

In the case of a complex K3 surface $S$, the global Torelli theorem explains how to recover automorphisms of $S$ from isometries of the intersection lattice. We consider automorphisms of odd prime order which are non-symplectic (i.e.\ whose action on $H^{2,0}(S)$ is not trivial), since their coinvariant lattice inside $H^2(S,\IZ)$ is indefinite. \Cref{thm: classific unimod,thm:unimodular_conjugacy} can be combined to give the following classification.

\begin{theorem}\label{thm:K3-classification}
Let $p$ be an odd prime number and $r,a$ nonnegative integers.
There exists a K3 surface $S$ and a non-symplectic automorphism $\sigma \in \aut(S)$ of odd prime order $p$
with invariant lattice $H^2(S,\ZZ)^\sigma$ $p$-elementary hyperbolic of rank $r \geq 1$ and discriminant $p^a$
if and only if
\begin{enumerate}
\setlength{\itemsep}{5pt}
\item  $p \leq 19$ and $22-r \equiv 0 \pmod{p-1}$;
\item  $0 \leq a \leq \min\{r, \frac{22-r}{p-1}\}$ and $a \equiv \frac{22-r}{p-1} \pmod{2}$;
\item if $a=0$ or $a=r$, then $r \equiv 2 \pmod{8}$.
\end{enumerate}
Moreover, the triple $(p,r,a)$ determines the pair $(S, \langle\sigma\rangle)$ up to deformation.
\end{theorem}

\begin{remark}
A study of non-symplectic automorphisms of odd prime order on K3 surfaces was already conducted by Artebani, Sarti and Taki in \cite{autom_k3_ord3} and \cite{autom_k3}, where all possible isometry classes of the invariant and coinvariant lattices inside the second cohomology lattice are listed and related to the topology of the fixed locus of the automorphism.
In a second step the authors prove that a K3 surface $S$ with a non-symplectic prime order automorphism $\sigma$ of given invariant lattice belongs to an explicit irreducible family. Since the action on the cohomology lattice does not vary within the family, this shows that the action is unique up to conjugacy. We can follow the converse approach: since the fixed locus is a deformation invariant of the pair $(S,\sigma)$, \Cref{thm:K3-classification} gives another explanation why the fixed locus is determined by the invariant lattice.
\end{remark}

In higher dimension, K3 surfaces generalize to irreducible holomorphic symplectic (ihs) manifolds. The cohomology group $H^2(X,\IZ)$ of an ihs manifold $X$ admits an integral $\ZZ$-lattice structure (by use of the Beauville--Bogomolov--Fujiki quadratic form).
The known deformation types of ihs manifolds are: $\hskn$ (i.e.\ deformation equivalent to Hilbert schemes of $n$ points on a K3 surface \cite[\S 6]{beauville}), $\kumn$ (i.e.\ deformation equivalent to the $2n$-dimensional generalized Kummer variety of an abelian surface \cite[\S 7]{beauville}), OG$_6$ and OG$_{10}$ (i.e.\ deformation equivalent, respectively, to the six-dimensional and ten-dimensional ihs manifolds first constructed by O'Grady \cite{ogrady6, ogrady10}). Whether or not there are more deformation types remains a wide open question.

Beyond the case of K3 surfaces the classification of non-symplectic automorphisms is not as immediate. The reason is twofold. Firstly, the second cohomology lattice is no longer unimodular. Secondly, instead of the orthogonal group of this lattice one has to consider a subgroup, the monodromy group, for geometric purposes.
Thirdly, birational geometry enters the picture.
We also remark that the natural map $\rho_X \colon \aut(X) \rightarrow O(H^2(X,\IZ))$ is injective for manifolds of type $\hskn$ and OG$_{10}$, but not in general. We will be interested in studying automorphisms $f \in \aut(X)$ whose action on cohomology $f^* \in O(H^2(X,\IZ))$ is an isometry of odd prime order.

We obtain the following generalization and unification of the classification results in small dimension of \cite{CC}, \cite{mongardi_tari_wandel:kummer} and \cite{grossi:nonsymplectic}. The case of OG$_10$ is treated for the first time.
\begin{theorem} \label{thm: classification k3n}
Let $X'$ be an ihs manifold of type $\hskn$, $\kumn$, OG$_6$ or OG$_{10}$.
Let $K$ be a primitive sublattice of $\Lambda\cong H^2(X',\ZZ)$. Then there exist an ihs manifold $X$ of the same deformation type of $X'$, a marking $\eta: H^2(X,\IZ) \rightarrow \Lambda$ and a non-symplectic automorphism $\sigma \in \aut(X)$ with $\sigma^* \in O(H^2(X,\IZ))$ of odd prime order $p$ such that $\eta(H^2(X,\IZ)_{\sigma^*}) = K$ and $\eta(H^2(X,\IZ)^{\sigma^*}) = K^\perp$ if and only if $K$ is $p$-elementary of discriminant $p^a$ and signature $(2, (a+2m)(p-1)-2)$ for some nonnegative integers $a, m$.
\end{theorem}
\begin{remark}
 The introduction of hermitian lattices to the study of automorphisms
 allows for an approach which is at large independent of the deformation type of the manifolds.
 Nevertheless we have to restrict to the known types in the hypothesis since we need an explicit description of the group of monodromy operators.
 Once the monodromy group is available the methods can be used for ihs manifolds of arbitrary deformation type.
\end{remark}

While reasonably short, the theorem leaves much to be desired. First of all it does not tell us for which values of $a$, $m$ (and $n$) such a primitive sublattice $K$ actually exists. Secondly, what does it mean geometrically for two automorphisms to have isomorphic (co)invariant lattices?

This is addressed in \Cref{sec: autosihs} where we study deformations and conjugations by birational morphisms of pairs $(X, G)$, with $X$ an ihs manifold and $G$ a group of automorphisms. We do so by extending the results of \cite{joumaa:order2} and \cite{bcs:ball} on moduli spaces of $(\rho, T)$-polarized ihs manifolds of type $\hskn$ to arbitrary ihs manifolds.

In \Cref{thm:non-symplectic-classification} we prove that, for non-symplectic automorphisms with an action on cohomology of odd prime order, the equivalence classes of pairs $(X, G)$ up to deformation and birational conjugation correspond to the conjugacy classes of the corresponding monodromy operators.

Then, in \Cref{thm:fiber} we show that the conjugacy class of such a monodromy is determined by the conjugacy class of an isometry $f$ (of the same order) of a \emph{unimodular} lattice $M \supset H^2(X,\ZZ)$ together with the orbit of a primitive embedding of a suitable lattice of rank one or two inside the invariant lattice $M^f$.

These primitive embeddings are studied and classified in  \Cref{sec: primitive vectors,sec: OG10}, where we also provide necessary and sufficient conditions for the existence of a primitive vector of given square and divisibility in a $p$-elementary lattice $M^f$ (\Cref{thm: existence for p-elem}).
This results in an explicit classification of the invariant and coinvariant lattices in $H^2(X,\ZZ)$, i.e. it answers the first question.

Computing the number of orbits of primitive vectors vectors up to $O(M^f)$ and $SO(M^f)$ (\Cref{thm:orbits}) gives the following theorem.
\begin{theorem}
Let $X$ be an ihs manifold and $G \leq \Aut(X)$ a group of non-symplectic automorphisms with $\ker \rho_X \leq G$ and $\rho_X(G)$ of odd prime order $p$. If $X$ is of type OG$_{10}$, then $(X,G)$ is determined up to deformation and birational conjugation by the isomorphism classes of the lattices $H^2(X,\ZZ)^G$, $H^2(X,\ZZ)_G$ if and only if $p\neq 23$. For manifolds of type $\hskn$ and $\kumn$, the lattices determine $(X,G)$ up to deformation and birational conjugation except for $p=23$ and the cases in \Cref{tbl:ambiguous-k3n,tbl:ambiguous-kumn}.
\end{theorem}

By combining all these results we explicitly classify non-symplectic automorphisms of the known ihs manifolds with an action on cohomology of odd prime order up to deformation and birational conjugation. (See \cite{grossi:nonsymplectic} for OG$_6$.)

\subsection*{Outlook}
It should be possible to extend the results to automorphisms of order $m>2$ of ihs manifolds with trivial action on the Picard group. These correspond to isometries with minimal polynomial $\Phi_1(x)\Phi_m(x)$. In another direction we expect that our methods generalize to isometries of $p$-elementary lattices and thus apply to automorphisms of supersingular K3 surfaces in positive characteristic. Is it possible to go beyond the $p$-elementary case?
What about isometries of prime order in lattices over number fields in the spirit of \cite{kirschmer:unimodular}?
Can we follow a similar strategy for automorphisms of singular symplectic varieties?

\subsection*{Acknowledgements} The authors thank Samuel Boissi\`ere, Chiara Camere, Alessandra Sarti and Davide Veniani for sharing their insights, Markus Kirschmer for explaining his work and how to prove \Cref{prop:genus2}, Rainer Schulze-Pillot for pointing out the work of Quebbemann, suggesting the use of M\"obius inversion and the theta series of $F_{23b}$.
Part of this work arose at the 2019 Japanese-European Symposium on Symplectic Varieties and Moduli Spaces at ETH Z\"urich.
\blfootnote{
A.~C. is grateful to Max Planck Institute for Mathematics in Bonn for its hospitality and financial support. A.~C. is supported by the Deutsche Forschungsgemeinschaft (DFG, German Research Foundation) under Germany's Excellence Strategy -- GZ 2047/1, Projekt-ID 390685813.}
\blfootnote{S.~B. is supported by  SFB-TRR 195 ``Symbolic Tools in Mathematics and their Application'' of the German Research Foundation(DFG).}

\addtocontents{toc}{\protect\setcounter{tocdepth}{2}}

\section{Prime order isometries of unimodular lattices}
Our path towards a classification of prime order isometries starts from the invariant and coinvariant lattices. We already remarked that the coinvariant lattice $L_f$ of an isometry $f \in O(L)$ of odd prime order $p$ can be seen as a hermitian lattice over $\ZZ[\zeta_p]$, whose classification is well understood. Assuming that $L$ is unimodular gives us sufficient information to characterize $L^f$ and $L_f$. In a final step we check when and how $L^f$ and $L_f$ glue to obtain a unimodular lattice. 

\subsection{\texorpdfstring{$\ZZ$}{Z}-lattices}
We assume that the reader is familiar with the basic theory of $\ZZ$-lattices and refer to \cite{nikulin,conway_sloane,kneser} for further details and
definitions.
Let $(L,b)$ and $(L',b')$ be integral lattices. A homomorphism
$f \colon L \rightarrow L'$ of $\ZZ$-modules is a \emph{homomorphism of
lattices} if it satisfies $b(x,y)=b'(f(x),f(y))$ for all
$x,y \in L$. Automorphisms of lattices are called \emph{isometries} and monomorphisms \emph{embeddings}. An embedding is \emph{primitive} if its cokernel is torsion free.
The orthogonal group $O(L,b)$ consists of the
automorphisms of $(L,b)$.
If $b$ is understood, we may abbreviate $(L,b)$ to $L$ and $b(x,x)$ to $x^2$.
In this case we write $L(-1)$ for $(L,-b)$. The genus of a lattice $L$ is denote by $\mathfrak{g}(L)$. We use Conway-Sloane's notation for genera of lattices (see \cite[Ch.\ 15]{conway_sloane}).

For a prime number $p$, a lattice is said to be $p$-elementary if its discriminant group is a sum of copies of $\ZZ/p\ZZ$. The following result gives a classification of $p$-elementary lattices, which we will need later on.
\begin{theorem}\cite[Ch.\ 15, Thm.\ 13]{conway_sloane}\label{thm:p-elementary}
Let $l_{\pm},n \in \NN$, $\epsilon \in \{\pm 1\}$.
Then the genus $\even_{(l_+,l_-)}p^{\epsilon n}$ is nonempty if and only if
$\epsilon$ is given for $\rk L \neq n$ by
\begin{equation}\label{eqn:p-sign}
 l_+ - l_- \equiv 2 \epsilon -2 -(p-1)n \pmod{8}
\end{equation}
 while if $\rk L =n$, then $\epsilon = \leg{-1}{p}^{l_-}$.
 
 The genus $\odd_{(l_+,l_-)}p^{\epsilon n}$ is nonempty if and only if
   $\epsilon=\leg{-1}{p}^{l_-}$ when $\rk L =n$.
\end{theorem}

\subsection{Hermitian lattices}
In this section we introduce hermitian lattices -- our key tool to study the coinvariant lattice.
Since this is the first systematic application of hermititan lattices in the context of automorphisms of ihs manifolds, we provide plenty of details.

\begin{definition}
Let $K$ be a field of characteristic zero and let $a \in K$.
Set $E = K[x]/(x^2-a)$ and denote by $\sigma\colon E \rightarrow E, [x] \mapsto [-x]$ the canonical involution.
A \emph{hermitian space} $(V,h)$ over $(E,\sigma)$ is a
finitely generated free $E$-module $V$ equipped with a nondegenerate $K$-bilinear form
\[h\colon V \times V \rightarrow E\]
which is $E$-linear in the first argument and satisfies
$h(x,y) = h(y,x)^\sigma$ for all $x,y \in V$.
Let $\ZZ_E$ be the maximal order of $E$.
A \emph{hermitian $\ZZ_E$-lattice} $(L,h)$ consists of a finitely generated $\ZZ_E$-module $L\subseteq V$ of full rank, equipped with the hermitian form $h$.
\end{definition}
If $E$ and $h$ are understood, we drop them from notation and simply speak of a hermitian lattice $L$.
For $a_1, \dots, a_n \in K$ we denote by $\langle a_1, \cdots, a_n \rangle$ a free $\ZZ_E$-module spanned by $e_1, \dots, e_n$
equipped with the hermitian form defined by $h(e_i,e_j)=\delta_{ij}a_i$.
\begin{remark}
 More generally one can define hermitian lattices over orders in algebras with involution.
 For instance for a finite group $G$ we can take $\ZZ[G]$ and the involution as inversion.
 If we assume the algebra to be \'etale, then the trace  construction (see below) works as well. But in full generality not much is known.
\end{remark}
\noindent\textbf{Notation for cyclotomic fields.}
Let $p$ be an odd prime and $\zeta$ a fixed primitive $p$-th root of unity. See \cite[Ch.\ 2]{washington:cyclotomic_fields} for details.
\begin{itemize}
 \item $\Phi_p(x)$ is the $p$-th cyclotomic polynomial.
 \item $E = \QQ[\zeta]$ and $K = \QQ[\zeta + \zeta^{-1}]$.
 \item $\Cl(E)$, $\Cl(K)$ are the respective class groups.
 \item $\Cl(E/K)=\Cl(E)/\Cl(K)$ is the relative class group.
 \item $\#\Cl(E/K)=\#\Cl(E) /\#\Cl(K)$ is the relative class number.
 \item $\zeta^{\sigma}=\zeta^{-1}$, i.e.\ the involution $\sigma$ is complex conjugation.
 \item $\ZZ_E = \ZZ[\zeta]$ and $\ZZ_K = \ZZ[\zeta + \zeta^{-1}]$.
 \item $\pi = (1-\zeta)$.
 \item $\mathfrak{p}=\pi \pi^\sigma \ZZ_K$.
 \item $\mathfrak{P}=\pi \ZZ_E$.
 \item $\mathfrak{A}= \mathfrak{D}^E_\QQ = \pi ^ {p-2} \ZZ_E$ is the absolute different.
 \item $\mathfrak{R}= \mathfrak{D}^E_K= \pi \ZZ_E$ is the relative different.
 \item $N = N^E_K$ and $T = \Tr ^E_K$ are the relative norm and trace.
 \item $\Omega(K)$ is the set of places of $K$.
 \item $\mathbb{P}(K)$ is the set of prime ideals of $K$.
\end{itemize}

\noindent\textbf{Notation concerning hermitian lattices.}
\begin{itemize}
 \item $(L,h)$ is a hermitian lattice.
 \item $(L,b=\Tr^E_\QQ \circ h)$ is its trace lattice.
 \item $\mathfrak{s}(L)=h(L,L)$ is the scale of $(L,h)$.
 \item $\mathfrak{n}(L)=\sum_{x \in L} h(x,x)\ZZ_K$ is the norm of $(L,h)$.
\end{itemize}

\subsection{The trace lattice}
If $(L,h)$ is a hermitian $\ZZ[\zeta]$-lattice define $b = \Tr^E_\QQ \circ\; h$. Then $(L,b)$ is a $\ZZ$-lattice, called the \emph{trace lattice}. Multiplication by $\zeta$ induces an isometry $f$ of $(L,b)$ with minimal polynomial $\Phi_p(x)$.
Conversely, if $(L,b)$ is a lattice and $f$ an isometry with minimal polynomial $\Phi_p(x)$, then $\zeta\cdot x = f(x)$ defines a $\ZZ_E$-module structure on $L$ and
\begin{equation}\label{eqn:hermitian_form}
h(x,y) = \sum_{i=0}^{p-1} b(x,f^i(y)) \zeta^i \in E
\end{equation}
\noindent defines a hermitian form.
We sum up this result in the following proposition. Note that it involves the choice of a fixed root of unity $\zeta$.
\begin{proposition}
The trace construction sets up an equivalence between the category of hermitian $\ZZ_E$-lattices $(L,h)$ and
the category consisting of pairs $((L,b),f)$, where $f \in O(L,b)$ is an isometry of minimal polynomial $\Phi_p(x)$ and morphisms are $f$-equivariant isometries.
\end{proposition}
We denote by 
\[L^\vee=(L,b)^\vee = \{ x \in L \otimes \QQ \mid b(x,L) \subseteq \ZZ  \}   \]
and
\[L^\#=(L,h)^\#= \{ x \in L \otimes \QQ \mid h(x,L) \subseteq \ZZ_E  \} \]
the respective \emph{dual lattices}.
They satisfy the following relation
\begin{equation}\label{eqn:discr}
(L,b)^\vee = \mathfrak{A}^{-1} (L,h)^\#.
\end{equation}

If $\mathfrak{B}L^\#=L$ for some fractional ideal $\mathfrak{B}$ of $E$, we call $(L,h)$ $\mathcal{B}$-modular. A $\ZZ_E$-modular hermitian lattice is called unimodular.

\begin{lemma}\label{lem:even}
 If $(L,h)$ is a hermitian lattice over a prime cyclotomic field such that its trace lattice $(L,b)$ is integral, then $(L,b)$ is even.
\end{lemma}
\begin{proof}
 That $(L,b)$ is integral is equivalent to $\scale(L) \subseteq \mathfrak{A}$. Let $x\in L$. Then $h(x,x) \in \norm(L) \subseteq K \cap \scale(L) \subseteq K \cap \mathfrak{A} = \mathfrak{D}^K_\QQ$. Thus $\Tr^E_\QQ(h(x,x))=2\Tr^K_\QQ(h(x,x)) \in 2\ZZ$.
\end{proof}

\begin{example}
 The root lattice $A_{p-1}$ admits a fixed point free isometry of order $p$. To see this, consider the extended Dynkin diagram $\widetilde{A_{p-1}}$. It is a regular polygon with $p$ vertices, with a rotational symmetry $f$ of order $p$ which fixes the sum of the $p$ vertices. This sum is precisely the kernel of $\widetilde{A_{p-1}}$. Thus $f$ descends to a fixed point free isometry of the quotient. But the quotient is isomorphic to $A_{p-1}$, therefore it can be seen as a hermitan $\ZZ[\zeta]$ lattice of rank one.
\end{example}

\subsection{Genera of hermitian lattices}
The classification of hermitian lattices largely parallels that of $\ZZ$-lattices. In this section we recall the parts of the classification needed for our purposes.
\begin{definition}
Two hermitian lattices $L$ and $L'$ are said to be in the same \emph{genus} if the completions
$L_\nu$ and $L'_\nu$ are isomorphic for all $\nu \in \Omega(K)$.
\end{definition}
Let $\nu \in \Omega(K)$ be a place of $K$.
The place $\nu$ is called \emph{good} if $E_\nu$ is either isomorphic to $K_\nu \times K_\nu$ or $E_\nu/K_\nu$ is an unramified field extension of degree $2$. Otherwise we call $\nu$ \emph{bad}, and then $E_\nu/K_\nu$ is a ramified extension of degree $2$.
A hermitian lattice over $E_\nu$ can be decomposed as an orthogonal direct sum of modular hermitian lattices. In the sequel we recall the classification of modular hermitian lattices \cite{jacobowitz}.
A unimodular hermitian lattice over $E_\nu$ with $\nu$ a good prime is unique up to isomorphism (cf.\ \cite[Prop.\ 3.3.5]{kirschmer:hermitian}).
In our case the only bad prime is $\pi$.
\begin{proposition}\label{prop:modular}\cite[Prop.\ 3.3.5]{kirschmer:hermitian}
Let $\pi$ be a bad prime which is coprime to $2$. Let $L$ be a
$\pi^i$-modular lattice of rank $r$ over $E_\pi$.
If $i$ is even, then \[L \cong \langle (\pi \pi^\sigma)^i,\dots, (\pi \pi^\sigma)^i, \det L (\pi \pi^\sigma)^{i(1-r)/2}\rangle.\]
If $i$ is odd, then
\[L \cong H_{i}^{\oplus r/2} = \left(
\begin{matrix}
0 & \pi^i \\
(\pi^\sigma)^i & 0\\
\end{matrix}
\right)^{\oplus r/2}.\]
\end{proposition}

Let $v_1, \dots , v_s$ be the real places of $K$.
Denote by $n_i$ the number of negative entries in a diagonal Gram matrix of $(L_{v_i},h_{v_i})$. They are called
the signatures of $(L,h)$.
Let $(s_+,s_-)$ be the signature of the trace lattice $(L,b)$ and $k^\pm_i$ the \emph{signatures} of $(L,b,f)$ as defined in the introduction.
For a place $q$ of $\QQ$, we obtain the orthogonal splitting
\[(L,h) \otimes \ZZ_q \cong \bigoplus_{\nu \mid q} (L,h)_\nu.\]
For $q=-1$ the infinite place, we obtain the orthogonal splitting of \cref{eqn:sign}.
By carrying out the trace construction for a hermitian lattice over $\CC/\RR$ of rank one we get that $k_i^-=2n_i$ and thus
\begin{equation*}
s_- = 2 \sum_{i=1}^s n_i.
\end{equation*}

The genus of a hermitian lattice $(L,h)$ is uniquely determined by its signatures $n_i$ and its modular decompositions at all primes dividing its determinant.
\begin{proposition}\cite[3.4.2 (3) and 3.5.6]{kirschmer:hermitian}\label{prop:existence-herm}
Given hermitian lattices $(L_\nu,h_\nu)$ at each place $\nu \in \Omega(K)$, all but finitely many of which are unimodular, there is a global hermitian lattice $(L,h)$ with
$(L,h)_\nu \cong (L_\nu,h_\nu)$ at all places if and only if
the set $S=\{ \nu \in \Omega(K) | \det (L_\nu,h_\nu) \not \in N(E^\times) \}$
is of even cardinality.
\end{proposition}
In order to analyse the classes in a genus more closely we need the following facts and notation:
\begin{itemize}
\setlength{\itemsep}{2pt}
 \item $\mathcal{I}$ the group of fractional ideals of $E$,
 \item $J = \{\mathfrak{A} \in \mathcal{I} \mid \mathfrak{A}\mathfrak{A}^\sigma = \ZZ_E\}$,
 \item $J_0 = \{e \ZZ_E \mid e \in E^\times \mbox{ with } ee^\sigma =1\}$,
 \item $C=\Cl(E)$ the class group of $E$,
 \item $C_0=\{ [\mathfrak{A}] \in C \mid \mathfrak{A}=\mathfrak{A}^\sigma\}$ the subgroup of $C$ generated by the image of $\Cl(K)$ and the prime ideals of $\ZZ_E$ ramified in $E/K$,
 \item $C/C_0 \rightarrow J/J_0$ induced by $\mathfrak{A} \mapsto \mathfrak{A}/\mathfrak{A}^\sigma$ is an isomorphism.
 \end{itemize}

In our case only a single prime ideal of $\ZZ_E$ ramifies in $E/K$, namely the principal ideal
$\mathfrak{P}=(1-\zeta)$. By \cite[Thm.\ 4.14]{washington:cyclotomic_fields} the natural
homomorphism $\Cl(K) \rightarrow \Cl(E)$ is injective.
This yields that $\#(C/C_0)=\#\Cl(E)/\#\Cl(K)=\#\Cl(E/K)$ is the relative
class number.

\begin{proposition}\label{prop:genus1}
Let $(L_1,h_1)$ and $(L_2,h_2)$ be hermitian lattices of rank one in the same genus.
Write $L_i=\mathfrak{I}_i x_i$ and $d_i=h_i(x_i,x_i)$ for $x_i \in L_i$ , $\mathfrak{I}_i \in \mathcal{I}$ and $i=1,2$. Let $n \in E$ with $n n^\sigma = d_1/d_2$.
Then $(L_1,h_1)$ is isometric to $(L_2,h_2)$ if and only if
\[\mathfrak{I}_1/\mathfrak{I}_2 \cdot n \in J_0.\]
Furthermore $\# \mathfrak{g}(L,h) = \# \Cl(E/K)$.
\end{proposition}
\begin{remark}
In terms of the isomorphism $J/J_0 \cong C/C_0$, this means that we find a fractional ideal $\mathfrak{I}$ of $E$ with $\mathfrak{I}_1/\mathfrak{I}_2 \cdot n =\mathfrak{I}/\mathfrak{I}^\sigma$. Then $\mathfrak{I}_1/\mathfrak{I}_2 \cdot n \in J_0$ if and only if $[\mathfrak{I}] \in C_0$.
\end{remark}
\begin{proof}
 Since $(L_1,h_1)$ and $(L_2,h_2)$ are in the same genus, the corresponding quadratic spaces $(V_i,h_i)$ are isomorphic. In particular $d_1/d_2$ is a norm. Hence we find $n$ as in the proposition which gives the isometry $g\colon (V_1,h_1) \rightarrow (V_2,h_2)$, $x_1 \mapsto n x_2$. The image of $L_1$ is $n \mathfrak{I}_1 x_2$. Now $g(L_1)$ and $L_2$ are isomorphic if and only if we find an isometry $e \in O(V_2,h_2)\cong J_0$ with $e n \mathfrak{I}_1 x_2 = \mathfrak{I}_2 x_2$, i.e. $\mathfrak{I}_1/\mathfrak{I}_2 \cdot n \in J_0$.
 Since the lattices in $\mathfrak{g}(L_2,h_2)$ are all of the form
 $(\mathfrak{A}L_2,h_2)$, $\mathfrak{A} \in J$, the number of classes in the genus is $\# J / J_0$.
\end{proof}
Next we use the case of rank one to settle the indefinite case of rank at least two. The crucial ingredient is the strong approximation theorem for unitary groups by Shimura \cite{shimura:unitary}.
\begin{definition}
Let $(L,h)$ be a hermitian lattice of rank $r$. Its determinant lattice $\det(L,h)$ consists of the module $\det L = \bigwedge^r L$ equipped with the hermitian form $\wedge^{n}h$.
Explicitly, if $L$ is given by the pseudo basis $L = \sum_{i=1}^r \mathfrak{A}_i x_i$ for fractional ideals $\mathfrak{A}_i$ of $E$ and $x_i \in L$, then $\det L = \prod_{i=1}^r \mathfrak{A}_i x$ where $x=x_1 \wedge \dots \wedge x_n$. Furthermore $h(x,x) = \det h(x_i,x_j)$.
\end{definition}
Let $L,L'$ be $\ZZ_E$ lattices in the $E$-vector space $V$.
The \emph{index ideal} is $[L: L']_{\ZZ_E}=\{\det \sigma | \sigma \in \Hom_{\ZZ_E}(L,L')\}$.
More concretely, let $L,L'\subseteq V=E^n$ with standard basis $e_1,\dots , e_n$ be given by pseudo bases $[\mathfrak{A}_i,x_i]$ and $[\mathfrak{B}_i,y_i]$. Then
$[L:L'] = \prod \mathfrak{A}_i / \mathfrak{B}_i \cdot (\det X / \det Y)$
where $X$ is the matrix with rows $x_i \in E^n$ and $Y$ the matrix with rows $y_i$.
We compute $\det L= \prod \mathfrak{A}_i \det(X) e_1 \wedge \dots \wedge e_n$
and $\det L'= \prod \mathfrak{B}_i \det(Y) e_1 \wedge \dots \wedge e_n$.
Clearly, \[[L: L']_{\ZZ_E}=[\det L': \det L]_{\ZZ_E}.\]

The following is inspired by Bayer-Fluckiger's use of determinant lattices in the
classification of unimodular hermitian lattcies \cite{fluckiger:unimodular}.
\begin{proposition}\label{prop:genus2}
 Let $E = \QQ[\zeta]$ and $K=\QQ[\zeta + \zeta^{-1}]$ where $\zeta$ is a primitive $p$-th root of unity and $p$ an odd prime.
 Let $(L,h)$ be a hermitian lattice over $E/K$. Suppose that $(L,h)$ is of rank at least two and indefinite.
 Then the number of classes in the genus $\mathfrak{g}(L,h)$ is the relative class number $\#C(E/K)$.
 Two lattices in $\mathfrak{g}(L,h)$ are isometric if and only if they have isometric determinant lattices.
\end{proposition}
\begin{proof}
Let $(L,h)$ be indefinite of rank two with ambient hermitian space denoted by $(V,h)$ and
 $\mathfrak{q} \in \mathbb{P}(K)$ be a prime. Recall the following definitions and facts from \cite{kirschmer:determinants}:
\begin{itemize}
\setlength{\itemsep}{2pt}
 \item $\mathcal{E}_0^\mathfrak{q}=\{u \in \ZZ_{E_\mathfrak{q}}^* \mid uu^\sigma =1\}$,
 \item $\mathcal{E}_1^\mathfrak{q}=\{u/u^{\sigma} \mid u \in \ZZ_{E_\mathfrak{q}}^*\} \subseteq \mathcal{E}_0^\mathfrak{q}$,
 \item $\mathcal{E}(L_\mathfrak{q})=\{\det(g) \mid g \in O(L_\mathfrak{q},h_\mathfrak{q}) \}\subseteq \mathcal{E}_0^\mathfrak{q}$,
 \item $P(L)=\{\mathfrak{q} \in \mathbb{P}(E) \mid \mathcal{E}(L_\mathfrak{q}) \neq \mathcal{E}_0^\mathfrak{q}\}$ consists only of primes ramified in $E/K$,
 \item $\mathcal{E}(L) = \prod_{\mathfrak{q} \in \mathbb{P}(L)} \mathcal{E}_0^\mathfrak{q}/\mathcal{E}(L_\mathfrak{q})$,
 \item $R(L)=\{(e \mathcal{E}(L_\mathfrak{q}))_{\mathfrak{q}\in P(L)} \in \mathcal{E}(L) \mid e \in \ZZ_E \mbox{ and } ee^\sigma=1\}$.
\end{itemize}

The number of special genera in the genus of $L$ is
$[C:C_0] [\mathcal{E}(L):R(L)]$ (cf. \cite[Lem.\ 4.6]{kirschmer:determinants}). Since the genus is indefinite, each special genus consists of a single isometry class. By definition, two lattices in $(V,h)$ which are isometric lie in the same special genus. Thus the number above is actually the number of classes in the genus.
As $\mathfrak{p}$ is the unique prime ramified in $E/K$,
$\mathcal{E}(L)=\mathcal{E}_0^\mathfrak{p}
/\mathcal{E}(L_\mathfrak{p})$.
By \cite[Thm.\ 3.7]{kirschmer:determinants},
$\mathcal{E}_1^\mathfrak{p} \subseteq\mathcal{E}(L_\mathfrak{p})$
and further the quotient $\mathcal{E}_0^\mathfrak{p}/
\mathcal{E}_1^\mathfrak{p}$ is of order $2$
generated by $-1 \cdot \mathcal{E}_1^\mathfrak{p}$
(cf. \cite[Lem.\ 3.5]{kirschmer:determinants}).
Since $(-1\cdot\mathcal{E}_1^\mathfrak{p})$ is clearly in $R(L)$, the index
$[\mathcal{E}(L):R(L)]=1$.
This implies the following: if $L' \subseteq V$ is a lattice in the genus of $L$, then it is isometric to $(L,h)$ if and only if the index $[L : L']_{\ZZ_E}=[\det L : \det L']_{\ZZ_E}$
lies in $J_0$, i.e.\ if and only if $\det (L,h)$ is isometric to $\det (L',h)$.
\end{proof}
\begin{lemma}
Suppose that the relative class number of $C(E/K)$ is odd. Then
\[\Xi\colon J/J_0 \rightarrow C(E/K),\quad \mathfrak{A}J_0 \mapsto [\mathfrak{A}]\]
is an isomorphism.
\end{lemma}
\begin{proof}
 Since in our setting $\# J/J_0 = \#\Cl(E/K)$, it suffices to show that $\Xi$ is injective.
 Let $(\ZZ_K^\times)^+$ be the set of totally positive units of $\ZZ_K$.
 Suppose that the relative class number is odd, then $(\ZZ_E^\times)^+= N(\ZZ_E^\times)$ by \cite[Prop.\ A.2]{shimura:abelian}. Let $\mathfrak{A}=a\ZZ_E \in J$ be principal. By the definition of $J$, $aa^\sigma \ZZ_E = \mathfrak{A}\mathfrak{A}^\sigma =  \ZZ_E$. Moreover, $aa^\sigma$ lies in $K$ and is totally positive since it is a norm. Thus $aa^\sigma$ lies in $ (\ZZ_K^\times)^+ = N(\ZZ_E^\times)$ and we can find  $b \in \ZZ_E^\times$ with $bb^\sigma = aa^\sigma$.
 Hence $\mathfrak{A} = a \ZZ_E = a/b \ZZ_E$ lies in $J_0$.
\end{proof}

 Let $(L,h)$ be a hermitian lattice. The \emph{volume} of $(L,h)$ is the fractional ideal
 $\mathfrak{v}(L)=[L^\# : L]_{\ZZ_E}$. Let $[\mathfrak{A}_i,e_i]$ be a pseudo basis for $L$.
Set $\mathfrak{I} = \prod_{i=1}^n \mathfrak{A}_i$ and $a = \det h(e_i,e_j)_{ij}$. Then the volume of $(L,h)$ is $\mathfrak{v}(L)= a  \mathfrak{I} \mathfrak{I}^\sigma$.
 Note that the class of $\mathfrak{I}$ in $C(E)$ is the Steinitz invariant of $L$.
\begin{lemma}\label{lem:odd-relative-class}
Suppose that the relative class number of $C(E/K)$ is odd.
Let $(L,h)$ be a hermitian lattice which is indefinite or of rank one.
Suppose that the volume $\mathfrak{v}(L)=\mathfrak{p}^l$ for some $l$ and let $L' \in \mathfrak{g}(L,h)$.
Then $L$ is isomorphic to $L'$ if and only if they have the same Steinitz class in $C(E/K)$.
\end{lemma}
\begin{proof}
 Since $\mathfrak{p}^l = a \mathfrak{I} \mathfrak{I}^\sigma$ is principal, the Steinitz class $[\mathfrak{I}]$ lies in the kernel of the norm map, i.e.\ it is an element of the relative class group. As $(L',h)$ runs through a set of representatives of the classes in the genus of $(L,h)$, its Steinitz invariant is given by $[\mathfrak{I}] [L:L'] \in C(E/K)$ where $[L:L']$ runs through a set of representatives of $J/J_0$. Since $\Xi$ is an isomorphism, the $[\mathfrak{I}] [L:L']$ give exactly the relative class group.
\end{proof}

\subsection{Fixed point free isometries of prime order}
Let $L$ be a unimodular lattice and $f \in O(L)$ of prime order $p$.
Since $L$ is unimodular, there is an $f$-equivariant isomorphism
of the discriminant groups $(L^f)^\vee / L^f \cong L_f^\vee / L_f$.
In particular, we see that $f$ acts trivially on the discriminant group of $L_f$.
This observation sets the topic of this section.

\begin{proposition}\label{prop:clas-fixed-pt-free}
	Let $p$ be an odd prime number and $(L,b)$ an integral lattice
	of signature $(s_+,s_-)$.
	Then some lattice in the genus of $(L,b)$ admits a fixed point
	free isometry $f$
	of order $p$ acting trivially on the
	discriminant group if and only if
	there are nonnegative integers $n,m \in \ZZ$ such that
	\begin{enumerate}
		\item[(i)]  $L$ is even, $p$-elementary of discriminant $p^n$,
		\item[(ii)] $s_+ + s_- = (n + 2m)(p-1)$,
		\item[(iii)] $s_+ \in 2\ZZ$.
	\end{enumerate}
	 Conversely, if the integers $s_+,s_-,n,m \in \ZZ_{\geq 0}$ satisfy \emph{(ii)}, \emph{(iii)} and
	 \begin{enumerate}
		\item[(iv)] if $n=0$, then $s_+ \equiv s_- \pmod{8}$,	  \end{enumerate}
	 then a triple $(L,b,f)$ with $L$ satisfying \emph{(i)} exists.

    Suppose that $(L,b)$ is indefinite or of rank $(p-1)$.
    Let $f$, $f'$ be fixed point free isometries or order $p$ of $L$ with associated hermitian lattices $(L,h)$, $(L,h')$. Then $f$ is conjugate to $f'$ if and only if they have the same signature and $\det(L,h)$ is isometric to $\det(L,h')$.
    The number of conjugacy classes with a given signature is either $0$ or $\#C(E/K)$.
    If moreover the relative class number is odd, then $\det(L,h)$ is isometric to $\det(L,h')$ if and only if $(L,f)$ and $(L,f')$ (seen as $\ZZ[\zeta_p]$-modules) have the same Steinitz class in $\Cl(E/K)$.
\end{proposition}
\begin{proof}
(i) Since $f=\zeta$ acts as the identity on the discriminant group and $(L,b)$ is integral, we have
\begin{equation}\label{eqn:trivial_act}
(1-\zeta) (L,b)^\vee \subseteq (L,b)\subseteq (L,b)^\vee.
\end{equation}
Thus $(L,b)^\vee / (L,b)$ is isomorphic to $(\ZZ[\zeta]/ (1-\zeta))^n$ for some $n \leq \dim_E L\otimes \QQ$.
Since  $\ZZ[\zeta]/ (1-\zeta) \cong \FF_p$ as abelian groups, we see that $(L,b)$ is $p$-elementary. Moreover
it is even by \Cref{lem:even}.

(ii) Recall that $\pi = (1-\zeta)$.
Using \cref{eqn:discr} we translate \cref{eqn:trivial_act} to
\begin{equation}\label{eq-disc_bound}
\pi^{3-p}(L,h)^\# \subseteq  (L,h) \subseteq \pi^{2-p} (L,h)^\#.
\end{equation}
Thus $(L,h)$ is unimodular at all primes $\mathfrak{q}$ of $\ZZ_K$ except $\mathfrak{p}=\pi \pi^\sigma \ZZ_K$.
Since the primes $\mathfrak{q}$ are good, a unimodular hermitian lattice over $\ZZ_{K_\mathfrak{q}}$ is uniquely determined by its rank.
By \Cref{eq-disc_bound} the modular decomposition of $L$ at $\mathfrak{p}$ may only have $\pi^i$ modular blocks for $2-p \leq i \leq 3-p$.
From the classification of $\pi^i$-modular hermitian lattices in  \Cref{prop:modular}, we extract that
\[(L,h)_\mathfrak{p} \cong M \oplus H_{2-p}^{\oplus m}\]
where $M$ is $\pi^{3-p}$-modular of rank $k$ and determinant $\epsilon \in K_\mathfrak{p}^\times / N(E_\pi^\times)$.
Since $\mathfrak{A}^{-1}H_{2-p}^\#$ equals $H_{2-p}$, we obtain
\[L^\vee/L=\mathfrak{A}^{-1} L^\# / L \cong \mathfrak{A}^{-1}M^\#/M \cong \ZZ_E/\mathfrak{P}^k\cong \FF_p^k\] giving $\det (L,b) = p^k$ and thus $k=n$.

(iii) This follows from the fact that $s_- = 2 \sum_{i=1}^s n_i$.

(iv) Recall from \Cref{prop:existence-herm}
that a collection $(L_{v},h_v)_{v\in \Omega(K)}$ of local
hermitian lattices, all but finitely many of which are
unimodular, comes from a single global hermitian lattice
$(L,h)$ if and only if the set
$S=\{v \in \Omega(K) | \det L_v \not \in N(E_v) \}$
is finite of even cardinality.
An infinite place $v$ lies in $S$ if and only if $n_v$ is odd.
Thus we obtain
\begin{equation}\label{eqn:parity}
\# S \equiv \sum_{i=1}^s n_i +
\begin{cases}
0 \mbox{ if } \epsilon (-1)^m \in N(E^\times_\pi)\\
1 \mbox{ if } \epsilon (-1)^m \not\in N(E^\times_\pi)
\end{cases} \pmod{2}
\end{equation}
where $\epsilon = 1$ if $n=0$.
We see that for $n\neq 0$ this condition uniquely determines
the norm class of $\epsilon$.
We show that for $n=0$ condition (iv)
is equivalent to $\#S \equiv 0 \pmod{2}$.
To this end note that $-1$ is a local norm at $\pi$
if and only if $p \equiv 1 \pmod{4}$.
Thus we can rewrite \cref{eqn:parity} with $\epsilon =1$ as
\[2 \#S \equiv 2 \sum_{i=1}^s n_i + m(p-1) \pmod{4}.\]
With $s_- = 2 \sum_{i=1}^s n_i$ and multiplying by $2$ we arrive at $2s_- \equiv 2m(p-1) \pmod{8}$. Subtracting $s_+ + s_- = (p-1)2m$ yields (iv).
Note that $n=0$ if and only if $(L,b)$ is unimodular, hence we need $s_+ - s_- \equiv 0 \pmod{8}$.
This settles the existence conditions.

The lattice $(L,b)$ determines $n$ and $m$. By the previous
considerations this determines the isomorphism class of $(L,h)_\nu$
at all finite places $\nu \in \mathbb{P}(K)$.
The signatures $n_i$ determine it at the infinite places.
Thus they give the genus of $(L,h)$.
By \Cref{prop:genus1,prop:genus2}
the classes in the genus of $(L,h)$ are determined by the determinant lattice, and
 for odd relative class number by the Steinitz class (cf.\  \Cref{lem:odd-relative-class}).\end{proof}

\begin{remark} \label{rmk: relative class numbers}
 The relative class number $h^{-}(\QQ[\zeta_p])$
 is $1$ for all primes $p\leq 19$. For $p=23$, $29$, $31$,
 $37$, $41$ it equals $3$,
 $2^3$,
 $3^2$,
 $37$,
 $11^2$ respectively. The next prime with \emph{even} relative class number is $p=113$ with  $h^-=2^3 \cdot 17 \cdot 11853470598257$ (cf.\ \cite[Tables \S 3]{washington:cyclotomic_fields}).
 \end{remark}
\subsection{The action on the discriminant group.}
For an integral $\ZZ$-lattice $(L,b)$ and $f \in O(L,b)$ fixed point free of order $p$, we denote by $O(L,b,f)$ the centralizer of $f$ in $O(L,b)$. Note that $O(L,b,f)= O(L,h)$, where $h$ is defined as in \cref{eqn:hermitian_form}. Let $\bar f$ be the isometry of $\disc{L}$ induced by $f$.
In what follows, we want to compute the image of the morphism
\[O(L,f) \rightarrow O(\disc{L},\bar f) \]
in the case $\bar f = \id_{\disc{L}}$.

Let $x \in L^\vee$ and $[x]=x + L$. If $b(x,x)\not \equiv 0 \mod \ZZ$, we have the reflection
\[\tau_{[x]}\colon L^\vee/L \rightarrow L^\vee/L, \quad y \mapsto
y -2\frac{b(x,y)}{b(x,x)}y.\]

This can be adapted for hermitian lattices.
For any $\delta \in \mathcal{E}_0=\{\delta \in \ZZ_E \mid \delta \delta^\sigma=1\}$ and $x \in V=L\otimes E$ with $h(x,x) \neq 0$ we obtain the quasi reflection
$\tau_{x,\delta} \in O(V,h)$ defined by
\[\tau_{x,\delta}: V \rightarrow V, \quad y\mapsto y + (\delta-1)\frac{h(y,x)}{h(x,x)}x.\]
It maps $x$ to $\delta x$ and it is trivial on $x^\perp$.

\begin{lemma}
 Let $x\in L^\vee\otimes \ZZ_p=L^\vee_\mathfrak{p}$ with $b(x,x) \not \equiv 0 \mod \ZZ_p$.
 Suppose that $\mathfrak{P} L_\mathfrak{p}^\vee \subseteq L_\mathfrak{p}$.
 Then
 $\tau_{x,-1}$ acts as the reflection $\tau_{[x]}$ on the discriminant group $L^\vee/L$.
\end{lemma}
\begin{proof}
 As $\Tr^{E_\mathfrak{p}}_{\QQ_p}(h(x,x)) = b(x,x) \not \equiv 0 \mod \ZZ_p$,
 we obtain that $h(x,x) \not \in \mathfrak{A}^{-1}$.
 On the other hand $h(x,x) \in \mathfrak{s}(L^\vee,h)_\mathfrak{p}= \mathfrak{P}^{-1} \mathfrak{A}^{-1}$.
 Thus $h(x,x)\ZZ_{E_\mathfrak{p}} = \mathfrak{s}(L^\vee)$
 and we can write $L^\vee = \ZZ_{E_\mathfrak{p}} x \oplus x^\perp$.
 This induces a compatible splitting on $\disc{L}$. Since $\pi x \in L$,
 we have $x\ZZ_{E_\mathfrak{p}}+L_\mathfrak{p} = x \ZZ_p+L_\mathfrak{p}$ and this splitting is
 $[x]\ZZ_p \oplus [x]^\perp$. We can now conclude by comparing the actions of $\tau_{x,-1}$ and $\tau_{[x]}$ on
 $[x]$ and $[x]^\perp$.
\end{proof}

\begin{lemma}
 The morphism
 \[O(L_\mathfrak{p},h_\mathfrak{p}) \rightarrow O(\disc{L})\]
 is surjective.
\end{lemma}
\begin{proof}
 Note that in our case $O(\disc{L})$ is just an orthogonal group over a field. Then the Cartan-Dieudonn\'e theorem \cite[Ch.\ 1, Thm.\ 5.4]{scharlau} says that it is generated by reflections. But these reflections lie in the image.
\end{proof}

\begin{corollary}
The special orthogonal group $SO(\disc{L})$ is contained in the image of $O(L,h) \rightarrow O(\disc{L})$.
\end{corollary}
\begin{proof}
 The special orthogonal group is generated by pairs
 $\tau_{[x]}\tau_{[y]}$. Then the corresponding isometry
 $t=\tau_{x,-1}\tau_{y,-1}$ has determinant one. By the strong approximation theorem, for any $k$ we can find $g \in O(L)$ with $(g-t)(L) \subseteq \mathfrak{p}^kL$. For us $k=1$ is enough.
\end{proof}

Since $[O(\disc{L}):SO(\disc{L})]$ is at most $2$, we need only one more generator to get full surjectivity.

\begin{proposition} \label{prop: surjectivity}
 Let $(L,h)$ be a hermitian $\ZZ_E$-lattice with $\mathfrak{P} L^\vee \subseteq L$.
 If $(L,h)$ is indefinite or of rank one, then the natural map
 \[O(L,h) \rightarrow O(\disc{L})\]
 is surjective.
\end{proposition}
\begin{proof}
 Set $V = L\otimes E$.
 If $(L,h)$ is of rank one, then $O(L^\vee/L)\subseteq\{\pm 1\}$ and the natural map is surjective.
 If $(L,b)$ is unimodular, the proposition is certainly true.
 Otherwise, by the proof of \Cref{prop:clas-fixed-pt-free}, $\scale(L) = \norm(L)\ZZ_{E}$. Let $x \in L$ be a local norm generator at $\mathfrak{p}$, that is, $h(x,x)\ZZ_{K_\mathfrak{p}} = \norm(L)_\mathfrak{p}$.
 Since $\norm(L_\mathfrak{p})\ZZ_{E_\mathfrak{p}}=\scale(L_\mathfrak{p})$, the reflection $\tau=\tau_{x,-1} \in O(V,h)$ satisfies $\tau(L_\mathfrak{p}) = L_\mathfrak{p}$.
 However, it has denominators at the finite set of primes
 \[Q = \{\mathfrak{q} \in \mathbb{P}(K) \mid \tau(L_\mathfrak{q}) \neq L_\mathfrak{q}\}.\]
 
 We shall use the strong approximation theorem (\cite[5.12]{shimura:unitary}, \cite{kneser:strong-approx}) in the formulation of \cite[Thm.\ 5.1.3]{kirschmer:hermitian} to compensate the denominators.
 Take $S = \mathbb{P}(K)$. Since $V$ is indefinite $\Omega(K) \setminus S$ contains an isotropic place. We set $T = Q \cup \{ \mathfrak{p}\}$, and define
 $\sigma_\mathfrak{q} = \tau^{-1} \circ \phi_\mathfrak{q} $ for $\mathfrak{q} \in Q$ where $\phi_\mathfrak{q}\in O(L_\mathfrak{q})$ is of determinant $-1$ (which is possible by \cite[Cor.\ 3.6]{kirschmer:determinants}). Finally set $\sigma_\mathfrak{p} = \id_{L_\mathfrak{p}}$.
 By the strong approximation theorem for any $k \in \NN$ we can find $\sigma \in O(V)$ with
 \begin{itemize}
  \item  $\sigma(L_\mathfrak{r}) = L_\mathfrak{r}$ for $\mathfrak{r} \in S \setminus T$ and
  \item  $(\sigma - \sigma_{q})(L_\mathfrak{q})\subseteq \mathfrak{q}^kL_{\mathfrak{q}}$ for $\mathfrak{q} \in T$.
 \end{itemize}
 
 Choose $k\geq 1$ large enough such that $\mathfrak{q}^k \tau( L_\mathfrak{q}) \subseteq L_{\mathfrak{q}}$ for all $\mathfrak{q} \in Q$.
 Then
 \begin{eqnarray*}
  \tau \circ \sigma (L_\mathfrak{q})
  &\subseteq& \tau \circ \left(\sigma - \sigma_\mathfrak{q}\right)(L_\mathfrak{q}) + \tau \circ \sigma_\mathfrak{q}(L_\mathfrak{q})\\
  &\subseteq& \tau(\mathfrak{q}^k L_\mathfrak{q}) + \tau\circ \sigma_\mathfrak{q}(L_\mathfrak{q})\\
  &\subseteq& \mathfrak{q}^k \tau(L_\mathfrak{q}) + \phi_\mathfrak{q}(L_\mathfrak{q})\\
  &\subseteq& L_\mathfrak{q}.
 \end{eqnarray*}

 Hence $\tau \circ \sigma$ preserves $L_\mathfrak{q}$ for all $\mathfrak{q} \in \mathbb{P}(K)$. As it is moreover an element of $O(V,h)$, it must be in $O(L,h)$.
 Since $k\geq 1$, both $\tau$ and $\tau \circ \sigma$ induce the reflection $\tau_{[x]}$ on the discriminant group. This reflection generates $O(\disc{L})/SO(\disc{L})$.
\end{proof}

For later use we prove the following Lemma.
\begin{lemma}\label{lem:SO}
 If $(L,h)$ is a hermitian $\ZZ_E$-lattice with trace lattice $(L,b,f)$, then $O(L,b,f)=SO(L,b,f)$.
\end{lemma}
\begin{proof}
 Let $f \in O(L,b,f)=O(L,h)$. When we view $f$ as an $E$-linear map its determinant $d=\det_E(f) \in E$ satisfies $dd^\sigma=1$. Viewed as a $\QQ$-linear map one obtains $\det_\QQ(f)=N^E_\QQ(\det_E(f))=N^K_\QQ\circ N^E_K(\det_E(f))=N^K_\QQ(1)=1$.
\end{proof}

\subsection{The classification.}
We now use the results of the previous section to obtain
existence and uniqueness results on prime order isometries of unimodular lattices. To make notation lighter, in the following we will denote by $A_L = L^\vee/L$ the discriminant group of an even $\ZZ$-lattice $(L,b)$ and by $q_L$ its discriminant quadratic form. The \emph{length} $l(A_L)$ is defined as the minimal number of generators of the group $A_L$.

\begin{proof}[Proof of \Cref{thm: classific unimod}]
If $L$ is unimodular, there exists an isometry
of order $p$ if and only if there exists a primitive sublattice
$S \subset L$ and a fixed point free isometry $f \in O(S)$ of order $p$
which acts trivially on the discriminant group of $S$.
Indeed, such an isometry $f \in O(S)$ glues with $\id_{S^\perp}$
(see \cite[Cor.\ 1.5.2]{nikulin}) to give an isometry of $L$
whose coinvariant lattice is $S$. By \Cref{prop:clas-fixed-pt-free},
the lattice $S$ is $p$-elementary of rank $(n+2m)(p-1)$ and
discriminant $p^n$, for some nonnegative integers $n,m$.
If we denote its signature by $(s_+,s_-)$, this gives
conditions (1), (2) and condition (5) for $n=0$.
By \cite[Thm.\ 1.12.2 (resp.\ 1.16.5)]{nikulin}, such a lattice $S$ embeds
primitively into some even (resp.\ odd) unimodular lattice
$L$ of signature $(l_+,l_-)$ if and only if (II) (resp. (I)), (3), (4) hold
(note that $l(A_S)=n$) and further
\[\text{if } n = l_+ + l_- - (n+2m)(p-1) \text{, then }
(-1)^{l_+ - s_+}p^n \equiv \text{disc}(K(q_S)) \mod \left( \ZZ_p^*\right)^2\]
where $K(q_S)$ is the unique $p$-adic lattice of rank $n$
and discriminant form $q_S$, i.e. the $p$-modular Jordan
component of $S \otimes \ZZ_p$. Its determinant is obtained
in \Cref{thm:p-elementary}
via $s_+ - s_- \equiv 2\epsilon  -2 - (p-1)n \pmod{8}$,
where $\epsilon \in \{\pm 1\}$ indicates
the unit-square class of the determinant.
The left side is computed by the Legendre symbol $\left(\frac{-1}{p}\right)^{l_+ - s_+} \equiv (p-1)(l_+-s_+) +1 \pmod{4}$. Inserting this for $\epsilon$, we arrive at
\[s_+ - s_- \equiv (p-1)(2l_+ - 2s_+ -n) \equiv (p-1)(l_+ - l_- - s_+ + s_-) \equiv (1-p)(s_+ - s_-)\pmod{8}\]
which gives (5). (Note that in the odd case (5) implies condition (1) of \cite[Thm 1.16.5]{nikulin}).
\end{proof}

\begin{proof}[Proof of \Cref{thm:unimodular_conjugacy}]
Let $L$ be a unimodular lattice and $f,g\in O(L)$ prime order
isometries with $L_f, L_g$ indefinite or of rank $p-1$. Suppose that $L^f\cong L^g$, that the signatures agree and that the determinant lattices of $L_f$ and $L_g$ are isometric.
We prove that $f$ and $g$ are conjugated as isometries of $L$.
By assumption there is an isometry $u\colon L^f \rightarrow L^g$, and by \Cref{prop:clas-fixed-pt-free} an isometry $v\colon(L_f,f|_{L_f})\rightarrow (L_g,g|_{L_g})$.
Let $\epsilon_f\colon A_{L_f} \rightarrow A_{L^f}$ and
$\epsilon_g\colon A_{L_g} \rightarrow A_{L^g}$ be the standard
isomorphisms between the discriminant groups of
orthogonal primitive sublattices inside a unimodular lattice.
By \Cref{prop: surjectivity} there exists an
isometry $w \in O(L_f)$ centralizing $f\vert_{L_f}$ whose action
on the discriminant group $A_{L_f}$ is
$\bar{w} = \bar{v}^{-1} \circ \epsilon_g^{-1} \circ \bar{u}
\circ \epsilon_f$. It follows from
\cite[Cor.\ 1.5.2]{nikulin} that
$u \oplus (w \circ v)\colon L^f \oplus L_f \rightarrow L^g
\oplus L_g$ extends to an isometry of $O(L)$ conjugating $f$
and $g$.
\end{proof}

By using \Cref{thm: classific unimod}, for a given genus $\even_{(l_+,l_-)}$ we can list all triples $(p,n,m)$ such that there exists a unimodular lattice $L \in \even_{(l_+,l_-)}$ and an isometry $f \in O(L)$ of odd prime order $p$ with $L_f$ of signature $(2, (n+2m)(p-1)-2)$ and discriminant $p^n$.
Notice that, while the lattice $L_f$ is uniquely determined (up to isometries) by the triple $(p,n,m)$, a priori there can be several distinct isometry classes in the genus of $L^f$. For geometric applications, we study the uniqueness of the invariant lattice for some selected unimodular genera.

\begin{lemma}\label{lem:uniqueness_L^G}
Let $L$ be a unimodular lattice in one of the genera
$\even_{(3,3)}, \even_{(4,4)}$, $\even_{(5,5)}, \even_{(3,19)},
\even_{(4,20)}, \even_{(5,21)}$. Let $f \in O(L)$
be an isometry of odd prime order $p$ with $L_f$
of signature $(2, (n+2m)(p-1)-2)$ and discriminant $p^n$,
for some $n,m \in \IZ_{\geq 0}$. Then the lattice $L^f$
is unique in its genus unless $L \in \even_{(4,20)}$
and $(p,n,m) = (23,1,0)$, where $L^f$ is isometric to either
$F_{23a} = \left(\begin{smallmatrix} 2 & 1 \\ 1 & 12 \end{smallmatrix}\right)$ or
$F_{23b} = \left(\begin{smallmatrix} 4 & 1 \\ 1 & 6 \end{smallmatrix}\right)$.
\end{lemma}

\begin{proof}
Since $L$ is unimodular, $q_{L^f} \cong -q_{L_f}$,
hence $L^f$ is also $p$-elementary. Denote by $(l_+, l_-)$
the signature of $L$. If $L^f$ is indefinite and $\rk L^f \geq 3$,
then $L^f$ is unique in its genus by \cite[Ch.\ 15, Thm.\ 14]{conway_sloane}.
For $L$ as in the statement, the only cases where one of
these two conditions fails are the following:
$(l_+,l_-,p,n,m) = (3,3,3,0,1)$, $(3,3,3,2,0)$, $(3,3,5,1,0)$,
$(4,4,3,1,1)$, $(4,4,7,1,0)$, $(4,20,3,1,5)$, $(4,20,23,1,0)$.
The genera of $L^f$ are respectively $\even_{(1,1)}$,
$\even_{(1,1)}3^{-2}$, $\even_{(1,1)}5^{-1}$,
$\even_{(2,0)}3^{-1}$, $\even_{(2,0)}7^1$,
$\even_{(2,0)}3^{-1}$, $\even_{(2,0)}23^1$.
By using \cite[Ch.\ 15, Tables 15.1 and 15.2a]{conway_sloane}
we check that there is only one isometry class
for each of these genera, except for $\even_{(2,0)}23^1$,
which contains the two distinct isometry classes
given in the statement.
\end{proof}

\section{Automorphisms of ihs manifolds}\label{sec: autosihs}
In this section we apply our results on isometries of unimodular lattices to obtain a classification of automorphisms of irreducible holomorphic symplectic manifolds of a given deformation type with an action of odd prime order on the second cohomology lattice.

\begin{definition}
An irreducible holomorphic symplectic (ihs) manifold is a compact complex K\"ahler simply connected manifold $X$ such that $H^0(X, \Omega^2_X) = \CC \omega_X$, where $\omega_X$ is an everywhere nondegenerate holomorphic $2$-form.
\end{definition}
If $X$ is ihs, then $H^2(X,\IZ)$ is torsion-free and
it is equipped with a nondegenerate symmetric bilinear form
of topological origin, which gives it the structure of an integral lattice of signature $(3,b_2(X)-3)$ (see \cite[Thm.\ 4.7]{fujiki_relation}).
For all known examples of ihs manifolds this lattice has been
computed explicitly: it is even and only depends on the
deformation type of the manifold. Let $H^2(X,\IZ) \cong \Lambda$
for some fixed lattice $\Lambda$. The choice of an isometry
$\eta: H^2(X,\IZ) \rightarrow \Lambda$ is called a \emph{marking} of
$X$ and two marked ihs manifolds $(X, \eta)$, $(X', \eta')$
are equivalent if there is an isomorphism
$f: X \rightarrow X'$ such that $\eta' = \eta \circ f^*$.

There exists a coarse moduli space $\mathcal{M}_\Lambda$ which
parametrizes (equivalence classes of) $\Lambda$-marked ihs manifolds
$(X, \eta)$ for $X$ of a fixed deformation type
(see \cite{huybr_basic}). Two points $(X, \eta)$ and
$(X', \eta')$ belong to the same connected component of
$\mathcal{M}_\Lambda$ if and only if $\eta' \circ \eta^{-1}$ is
a parallel transport operator (see \cite[\S 1.1]{markman}
for the definition).
Denote by $\Mo^2(X) \subset O(H^2(X,\IZ))$ the
\emph{monodromy group}, which is the group of
\emph{monodromy operators} of $X$, i.e.\ parallel transport
operators $\gamma\colon H^2(X,\IZ) \rightarrow H^2(X,\IZ)$ .
Let
\[ \Omega_\Lambda := \left\{ \CC\omega \in \mathbb{P}(\Lambda \otimes \CC) \mid (\omega, \omega) = 0, (\omega, \overline{\omega}) > 0\right\}\]
be the \emph{period domain} and define the period map 
\[ \mathcal{P}: \mathcal{M}_\Lambda \rightarrow \Omega_\Lambda, \quad (X,\eta) \mapsto \eta(H^{2,0}(X)).\]
By \cite[Thm.\ 5]{beauville} and \cite[Thm.\ 8.1]{huybr_basic},
$\mathcal{P}$ is a local homeomorphism and, for any connected
component $\mathcal{M}_\Lambda^0 \subset \mathcal{M}_\Lambda$, its
restriction $\mathcal{P}_0: \mathcal{M}_\Lambda^0 \rightarrow
\Omega_\Lambda$ is surjective. Moreover, for any $p \in \Omega_\Lambda$
the fiber $\mathcal{P}_0^{-1}(p)$ consists of pairwise
inseparable points and if $(X, \eta), (X', \eta') \in \mathcal{M}_\Lambda$
are inseparable, $X$ and $X'$ are bimeromorphic (global Torelli theorem; see \cite[Theorem 2.2]{markman}).

The \emph{K\"ahler cone} $\K_X \subseteq H^{1,1}(X,\RR)$ consists of the classes of K\"ahler metrics.
The \emph{positive cone} $\C_X$
is the connected component of $\{x \in H^{1,1}(X,\RR) \mid x^2 >0\}$ containing a K\"ahler class.
The positive cone admits two important wall and chamber decompositions.
We denote by $\Delta(X) \subset H^{1,1}(X,\RR) \cap H^2(X,\ZZ)$ the set of primitive integral \emph{monodromy birationally minimal} (MBM) classes (see \cite[Def.\ 1.13]{amerik_verbitsky_mbm}) and by $\B \Delta(X)\subseteq \Delta(X)$ the set of \emph{stably prime exceptional divisors} of $X$. Then the chambers of
\[\C_X \setminus \bigcup_{\delta \in \Delta(X)} \delta^\perp\]
are called the \emph{K\"ahler type} chambers. One of them is the K\"ahler cone.
The chambers of
\[\C_X \setminus \bigcup_{\delta \in \B \Delta(X)} \delta^\perp\]
are called the \emph{exceptional chambers}. The exceptional chamber containing the K\"ahler cone is 
the \emph{fundamental exceptional chamber} $\mathcal{FE}_X$.
Its closure equals the closure of the birational K\"ahler cone.
An element $\delta$ of $\B \Delta(X)$ defines the reflection
$\tau_\delta(x)\defeq x - 2(x,\delta)/(\delta,\delta) \delta$,
which turns out to be a monodromy (see \cite[\S 6.2]{markman}). The group
$W_{Exc}(X) \leq \Mo^2(X)$ generated by these reflections
is called the Weyl-group.
It acts transitively
on the set of exceptional chambers and the fundamental exceptional
chamber is a fundamental domain for the action of
$W_{Exc}(X)$ on $\C_X$.

\smallskip
\emph{Fixing a connected component} $\M^\circ_\Lambda$ of
$\M_\Lambda$  one can transport these objects to $\Lambda$,
so that we may speak of $\Delta(\Lambda)$, $\B\Delta(\Lambda)$, $\C_\Lambda$,
$\Mo^2(\Lambda)$ and so on. For instance given
$(X,\eta) \in \M_\Lambda^\circ$ we have
$\Delta(X) = \eta^{-1}(\Delta(\Lambda)) \cap H^{1,1}(X,\RR)$.

The importance of theses cones for our work lies in Verbitsky's strong Torelli theorem \cite{verbitsky} for ihs manifolds as formulated by Markman.
\begin{theorem}\cite[Thms.\ 1.2, 1.6, Cor.\ 5.7]{markman}\label{thm:torelli}
Let $X_1$, $X_2$ be irreducible holomorphic symplectic manifolds and $\phi\colon H^2(X_1,\ZZ) \rightarrow H^2(X_2,\ZZ)$ a parallel transport operator which is a Hodge isometry.
\begin{itemize}
 \item There exists an isomorphism $f\colon X_2 \rightarrow X_1$ with $\phi=f^\ast$ if and only if $\phi(\K_{X_1})=\K_{X_2}$.
 \item If $X_1$, $X_2$ are projective, there exists a birational map $f\colon X_2 \dashrightarrow X_1$ with $\phi=f^\ast$  if and only if $\phi(\mathcal{FE}_{X_1})=\mathcal{FE}_{X_2}$.
\end{itemize}
\end{theorem}

Thus birational models of a projective ihs manifold $X$ are given by the K\"ahler type chambers contained in the fundamental exceptional chamber.
The natural homomorphism
\[\rho_X \colon \aut(X) \rightarrow \Mo^2(X), \quad \sigma \mapsto (\sigma^*)^{-1}\]
has finite kernel (see \cite[Prop.\ 9.1]{huybr_basic}).
Its image is computed by the strong Torelli theorem above.

Let $\sigma \in \aut(X)$ be a (biholomorphic) automorphism of an ihs manifold $X$.
The automorphism $\sigma$ is said to be symplectic if $\sigma_\CC^*(\omega_X) = \omega_X$, non-symplectic otherwise. It can be readily checked (cf. \cite[\S 5]{smith})
that the coinvariant lattice $H^2(X,\IZ)_{\sigma^*}$ of a
symplectic automorphism $\sigma \in \aut(X)$ is negative
definite and contained in $\ns(X)$. On the other hand,
if $\sigma$ is non-symplectic then $H^2(X,\IZ)_{\sigma^*}$
has signature $(2,\ast)$ and it contains the transcendental
lattice of $X$.

\subsection{Deformations}\label{sect:deformation}
In this section we apply results by Horikawa to show the existence of a universal deformation for pairs $(X,f)$ consisting of an ihs manifold and an automorphism. The result is surely known to the experts but for lack of a reference we give a detailed proof.

\begin{definition}
 Let $X$ be an ihs manifold and $f\in \Aut(X)$. A deformation of the pair $(X,f)$ consists of a smooth proper holomorphic map $p \colon \X \rightarrow B$, $0\in B$ with a distinguished fiber $\X_0 \defeq p^{-1}(0)=X$ and an automorphism $F \in \Aut(\X/B)$ such that $F|_{\X_0}=f$.
\end{definition}

From a morphism $s \colon B' \rightarrow B$, we obtain a family of deformations $\X'=\X \times_B B' \rightarrow B'$ and $F' = F \times \id_{B'}$. We call it the family induced from $(\X,B,p,F)$ via $s$.

In the following only the germ at $0 \in B$ of a deformation is of relevance and all statements are to be read in this sense.

\begin{definition}
 A deformation $\mathcal{B}=(\X,B,p,F)$ of $(X,f)$ is called \emph{versal}, if for every deformation $\mathcal{B}'=(\X',B',p',F')$ there exists a morphism $s\colon B' \rightarrow B$ such that $\mathcal{B}'$ is induced by $\mathcal{B}$ via $s$.
 If moreover $s$ is unique, then the deformation is called \emph{universal}.
\end{definition}

In \cite[\S 4]{bcs} the authors construct a family of deformations of the pair $(X,f)$.
This family will later turn out to be universal. We review the construction.
Let $p \colon \mathcal{X} \rightarrow D=\mbox{Def}(X)$ be the universal family of deformations $X$. By \cite[Thm 8.1]{horikawa:III} we obtain (the germ of) a family
$\mathcal{X}' \rightarrow D$ of deformations of $X$ and a holomorphic map $ \mathcal{X}\rightarrow \mathcal{X}'$ whose restriction to $X$ coincides with $f$. By the universality of $p$ we obtain $\mathcal{X}'$ as a pullback of $\mathcal{X}$ under a unique map $\gamma \colon D \rightarrow D$.
\[
\begin{tikzcd}
 \X \arrow[r," "]\arrow[d,"p"] & \X'=\X \times_\gamma D  \arrow[r,""]\arrow[d,"p"]& \X \arrow[d,"p"]\\
 D \arrow[r,"\id"]& D \arrow[r,"\gamma"] &  D
\end{tikzcd}
\]

By composition we have a map $\Phi \colon \mathcal{X} \rightarrow \mathcal{X}$.
Then the restriction of $\Phi$ to $D^\gamma= \{d \in D \mid \gamma(d)=d\}$
gives a family of deformations of $(X,f)$. Note that $D^{\gamma}$ is smooth.

\begin{proposition}
The restriction
$F \colon \mathcal{X}|_{D^\gamma} \rightarrow \X|_{D^\gamma}$ of $\Phi$ is the universal deformation of $(X,f)$.
\end{proposition}\label{prop:universal}
To prove the versality, we will use the results and language of \cite{horikawa:I}. For this purpose we set
$\Phi_0  = \Phi|_X \colon X \rightarrow \X$ and $\hat{\Phi}=(\Phi,p)\colon \X \rightarrow \X \times D, x\mapsto (\Phi(x),p(x))$.
Then $(\X,\hat{\Phi},p,D)$ is a deformation of $\Phi_0$ in the sense of \cite[Def.\ 1.1]{horikawa:I}.

\begin{lemma}
 The deformation $(\X,\hat{\Phi},p,D)$ is a versal deformation of $\Phi_0$ in the sense of \cite[Def.\ 1.2]{horikawa:I}.
\end{lemma}
\begin{proof}
This follows from \cite[Thm.\ 2.1]{horikawa:I}. It applies if we check that the so called characteristic map of $(\Phi,p)$ is surjective.
For this purpose let $i\colon X \rightarrow \X$ be the inclusion and consider the exact sequences
\begin{equation*}
\begin{tikzcd}
 0 \arrow[r," "] & \T_X \arrow[r,""]\arrow[d,"\id"]& i^*\T_{\X} \arrow[d,"df"] \arrow[r]& p^*\T_D \cong T_0D \otimes_\CC \O_X  \arrow[r] \arrow[d,""] &0 \\
 0 \arrow[r," "] & \T_X \arrow[r,""]& \Phi_0^*\T_{\X} \arrow[r]& \L \arrow[r] & 0
\end{tikzcd}
\end{equation*}
where $\L$ is defined as the cokernel of the lower part. Note that the vertical arrows in the middle are isomorphisms.
We take the corresponding long exact sequence in cohomology and obtain a commutative diagram
\begin{equation*}
\begin{tikzcd}
 H^0(X,p^*\T_D) = T_0 D \arrow[r,"\rho"] \arrow[d,"\tau"]& H^1(X,\T_X) \arrow[d,"\id"]\\
 H^0(X,\L)\arrow[r,"\delta "] & H^1(X,\T_X)
\end{tikzcd}
\end{equation*}
where $\rho$ is the Kodaira-Spencer map of
$p\colon \X \rightarrow D$. Since $p$ is the universal
deformation, $\rho$ is an isomorphism, and $\tau$ is an
isomorphism by definition.
By \cite[Prop.\ 1.4]{horikawa:I} $\tau = \delta \circ \rho^{-1}$ is the characteristic map.
\end{proof}

\begin{proof}[Proof of \Cref{prop:universal}]
Let $\X' \rightarrow B, F'\in \Aut(\X'/B)$ be a family of deformations of $(X,f)$.
By the universality of $p$, we obtain a diagram
\[
\begin{tikzcd}
 \X' \arrow[r," F'"]\arrow[d,"p'"] & \X'=\X \times_D B  \arrow[r,"p_\X"]\arrow[d,"p'"]& \X \arrow[d,"p'"]\\
 B \arrow[r,"\id"]& B \arrow[r,"s"] &  D
\end{tikzcd}
\]
and the family $\hat{\Phi}'= (p_\X \circ F',p')\colon \X' \rightarrow \X \times B$ of deformations of $\Phi_0$.
By the versality of $\hat{\Phi}$, we obtain that
$\hat{\Phi}' = \hat{\Phi} \times \id_B \colon \X' \rightarrow (\X \times D) \times_D B = \X \times B$.
Note that in fact $\hat{\Phi}'$ factors through $\X \times_D B =\X'$. This is summarized in the following diagram:
\[
\begin{tikzcd}
 \X\times_D B \arrow[rrrr, "\hat{\Phi}'=\hat{\Phi} \times \id_B"]\arrow[ddrr]\arrow[dddd] \arrow[drr,"F'"]&  &  &  & \X \times B \arrow[lldd]\arrow[ddd]\\
  &  &\X\times_D B \arrow[d] \arrow[urr]&  & \\
     &  &B \arrow[dl,"s"]\arrow[dr,"s"] &  &\\
     &D \arrow[rr,"\gamma"]&  &D & \X \arrow[l]\\
 \X \arrow[ur]\arrow[rrrr,"\hat{\Phi}"] &  &  &  & \X \times D \arrow[u]\arrow[lllu]
\end{tikzcd}
\]

Looking at various subdiagrams, one can check that the diagram is actually commutative.
In particular $\gamma \circ s = s$. Hence $s(B) \subseteq D^\gamma$. Thus it makes sense to restrict everything to the fibers over $D^\gamma$, and $\X' = \X|_{D^\gamma} \times_{D^\gamma} B$.
We obtain the diagram
\[
\begin{tikzcd}
 \X' \arrow[ddd]\arrow[rr,"F'"]\arrow[bend left=20, rrr,"\hat{\Phi}'"]\arrow[dr] & & \X'\arrow[dl] \arrow[ddd] \arrow[r]&\X|_{D^\gamma} \times B \arrow[ddd]\\
 & B \arrow[d,"s"]&  & &\\
 & D^\gamma&  & &\\
 \X|_{D^\gamma} \arrow[bend right=20,rrr,"\hat{\Phi}|_{D^\gamma}"]\arrow[ur]\arrow[rr,"F"] & & \X|_{D^\gamma} \arrow[lu]&\X|_{D^\gamma} \times D^{\gamma} \arrow[l]
\end{tikzcd}
 \]

\noindent where the inner quadrangle must commute since the outer one does. We infer that $F'=F \times \id_B$ is the pullback of $F$.
\end{proof}

If, instead of a single automorphism $f$, we have a group $G=\langle f_1,\dots f_n \rangle$, then we obtain an action of $G$ on $D$ via $\gamma_1, \dots , \gamma_n$. Now the universal family arises by restricting to the fixed locus $D^G$.
\subsection{Moduli spaces}
In this section we use universal deformations to construct a moduli space for pairs $(X,G)$.
In a second step we take a closer look at the period map, to determine where it is injective. This study is modeled on previous work in the case of K3 surfaces by Dolgachev, Kond\={o} \cite{dolgachev-kondo} and for ihs manifolds of type $\hskn$ by Joumaah \cite{joumaa:order2} and Boissi\`ere, Camere, Sarti \cite{bcs:ball}.

\begin{definition}
 Let $X,X'$ be ihs manifolds and $G \leq \Aut(X)$, $G' \leq \Aut(X')$. Then we call $(X,G)$ and $(X',G')$ \emph{bimeromorphically conjugate}
 if there exists a bimeromorphic map $\phi\colon X \dashrightarrow X'$ such that $\phi G \phi^{-1} = G'$.
 They are called \emph{deformation equivalent}, if there exists a
 connected family $\psi \colon\mathcal{X} \rightarrow B$ of ihs manifolds, a group of automorphisms $\mathcal{G}$ of $\mathcal{X}/B$ and two points $b,b'$ such that the restriction of $(\mathcal{X},\mathcal{G})$ to the fibers above $b$ and $b'$ give $(X,G)$ and $(X',G')$.
 \end{definition}

 Fix a connected component $\mathcal{M}_\Lambda^\circ$ of the moduli space of $\Lambda$-marked ihs manifolds of a given deformation type.
 Let $(X,\eta)$ be a marked pair belonging to this component. We define $\Mo^2_\circ(\Lambda) = \eta \Mo^2(X)\eta^{-1}$.
 This choice is independent of $(X,\eta)$ as long as they stay in the same connected component.
 A different connected component results in a conjugate subgroup of $O(\Lambda)$.

 \begin{definition}
 Let $H \leq O(\Lambda)$ be a subgroup. An $H$-marked ihs manifold is a triple $(X,\eta,G)$ such that $(X,\eta)$ is a marked pair, $\ker \rho_X \leq G \leq \Aut(X)$ and $\eta\rho_X(G)\eta^{-1}=H$.
 Two $H$-marked ihs manifolds $(X_1,\eta_1,G_1)$ and $(X_2,\eta_2,G_2)$ are called \emph{isomorphic} (respectively, \emph{bimeromorphic})
 if there exists an isomorphism (respectively, a bimeromorphic map) $f\colon X_1 \rightarrow X_2$ such that
 $\eta_1 \circ f^* = \eta_2$. In particular $f G_1 f^{-1} = G_2$.
 We call $H$ \emph{effective} if there exists at least one $H$-marked ihs manifold.
 \end{definition}

 Let $(X,\eta,G)$ be an $H$-marked ihs manifold.
 The action of $G$ on $H^{2,0}(X)$ induces via the marking a character $\chi \colon H \rightarrow \CC^*$.
 We can construct a coarse moduli space parametrizing isomorphism classes of $H$-marked ihs manifolds $(X,\eta,G)$ with
 $\chi(\eta\rho_X(g)\eta^{-1}) \cdot \omega_X = (g^*)^{-1} \omega_X$ for all $g \in G$, by gluing the base spaces of the universal deformations of $(X,\eta,G)$ constructed in \Cref{sect:deformation}.
 We denote the resulting moduli space by $\mathcal{M}^\chi_H$.

 \begin{proposition}\label{prop:forgetful}
 The forgetful map $\phi \colon \mathcal{M}^\chi_H \rightarrow \mathcal{M}_\Lambda, (X,\eta,G) \mapsto (X,\eta)$
 is a closed embedding.
 \end{proposition}
 \begin{proof}
 Let $(X,\eta,G)$ be an $H$-marked ihs manifold.
 The universal deformation $\X \rightarrow D$ of $X$ with marking induced by $\eta$ provides the neighborhood $D$ of the point $(X,\eta)$ in the moduli. Then $D^G= \im \phi \cap D$ is closed by its definition as a fixed point set of a finitely generated group of automorphisms.
 It remains to show that the map is injective.
 Indeed, since $\ker \rho_X \leq G$, we can reconstruct the group as $G = \rho_X^{-1}(\eta^{-1} H \eta)$.
 \end{proof}

 \begin{remark}
 Let $G$ be a group and $\rho \colon G \rightarrow O(\Lambda)$ a representation. Then one can define a marked $\rho$-polarized ihs manifold as a triple $(X,\rho',\eta)$ where $\rho'\colon G \rightarrow \Aut(X)$ is a homomorphism such that
 $\rho_X(\rho'(g)) = \eta^{-1} \rho(g) \eta$ for all $g \in G$.
 However, the forgetful map $(X,\rho',\eta)\mapsto (X,\eta)$ is in general not  injective in the presence of cohomologically trivial automorphisms.
 \end{remark}
The corresponding period domain is
\[\Omega_\Lambda^\chi=\{\CC \omega \in \Omega_\Lambda \mid h(\omega)=\chi(h)\cdot \omega \mbox{ for all } h \in H\}.\]
This results in a period map
$\P\colon \M^\chi_H \rightarrow \Omega_\Lambda^\chi$, which is a
local isomorphism by \Cref{prop:forgetful}.
Following \cite{joumaa:order2, bcs:ball} we exhibit
a bijective restriction of the period map.
We set $\mathcal{M}_H^{\circ,\chi} = \phi^{-1}(\mathcal{M}_\Lambda^\circ$).
Let
\[ \Lambda^\chi_\CC \defeq \{x \in \Lambda_\CC \mid h(x)=\chi(h)\cdot x \mbox{ for all } h \in H\}\]
and $N = (\Lambda_\CC^\chi)^\perp \cap \Lambda$.
For $H$-marked ihs manifolds we have
$N\subseteq \eta(\NS(X))$ with equality for a
very general subset of ihs manifolds.
Suppose that $\chi$ is nontrivial.
Then $N$ is of signature
$(1,t)$ for some $t$. In this case $H$-marked ihs manifolds are \emph{projective}, and we will speak of \emph{birational} instead of bimeromorphic maps to emphasize this. 

Let $\C_N = \C_\Lambda \cap N_\RR$.
We say that $(X,\eta,H)$ is $(H,N)$-polarized,
if $\C_N \cap \eta(\mathcal{K}_X) \neq \emptyset$.
Indeed, $(X,\eta^{-1}|_N)$ is an $N$-polarized ihs manifold
in the sense of \cite{camere:lattice-polarized-ihs}.

Recall that $\Delta(\Lambda)$ denotes the set of MBM-classes of $\Lambda$.
Set $M = N^\perp$ and $\Delta(M)=\Delta(\Lambda) \cap M$.
Let $\Delta'(M)$ consist of those $v \in \Delta(\Lambda)$ such that
$v = v_N + v_M$, $v_N \in N_\QQ$, $v_M \in M_\QQ$ with
$q(v_N)<0$.

 \begin{definition}
 Suppose that $\chi$ is nontrivial.
 Let $K(N)$ be a K\"ahler-type chamber of $\C_N$ preserved by
 $H$. An $(H,N)$-polarized ihs manifold $(X,\eta,G)$ is called
 \emph{$K(N)$-general}, if $\eta(\mathcal{K}_X) \cap N_\RR = K(N)$.
 It is called \emph{$H$-general}, if it is $K(N)$-general for
 some $K(N)$.
 \end{definition}
Note that $(X,\eta,G)$ is $H$-general if and only if
$\Delta'(M) \cap \eta(\NS(X))=\emptyset$.
 Let $\mathcal{M}^{\circ,\chi}_{K(N)}$ be the subset of
 $\mathcal{M}_H^{\circ,\chi}$ consisting of $K(N)$-general
 $(H,N)$-polarized manifolds. For $\delta \in \Delta(\Lambda)$,
 let $H_\delta \subset \PP(\Lambda_\CC)$ be the hyperplane
 orthogonal to $\delta$.
 We denote by $\Delta(K(N)) \subset \Delta'(M)$ the subset of elements $\delta$ such that $H_\delta \cap K(N) \neq \emptyset$.

 \begin{proposition}\label{prop:period-map}
  Suppose that $H\leq Mon^2(\Lambda)$ is finite and $\chi$ nontrivial. Set $\Delta = \bigcup_{\delta \in \Delta(M)} H_\delta$ and
  $\Delta' = \bigcup_{\delta \in \Delta(K(N))} H_\delta$.
  \begin{enumerate}
   \item The period map
  \[\mathcal{P} \colon \mathcal{M}^{\circ,\chi}_{H} \rightarrow \Omega^\chi_\Lambda \setminus \Delta\]
  is surjective.
  \item Its restriction
 \[\P_{K(N)}\colon\mathcal{M}^{\circ,\chi}_{K(N)} \rightarrow \Omega^\chi_\Lambda \setminus \left(\Delta \cup \Delta'\right)\]
 is bijective.
  \end{enumerate}
  Suppose in addition that $H$ is of prime order.
  \begin{enumerate}
   \item[(3)] Let $w \in \Omega^\chi_\Lambda \setminus (\Delta \cup \Delta')$
 and $(X_i,G_i,\eta_i)$ for $i=1,2$ be two $H$-general elements of the fiber $\P^{-1}(w)$. Then $(X_1,G_1)$ is birationally conjugate to $(X_2,G_2)$.
  \end{enumerate}

 \end{proposition}
\begin{proof}[Proof of (1) and (2)]
 We adapt the proofs of
 \cite[Thms. 4.5, 5.6]{bcs:ball} to our situation
 by using \Cref{prop:forgetful}.
 
 Let $(X,\eta,G)$ be an $H$-marked ihs manifold and $\delta \in \Delta(\Lambda)$. If
 $\P(X,\eta,G) \in H_\delta$, then
 $\eta^{-1}(\delta) \in \NS(X)$ is an MBM class of $X$,
 hence $\delta^\perp \cap \eta(\K_X) = \emptyset$.
 Since $\eta(\K_X) \cap N_\RR$ is nonempty, we obtain that
 $\delta \not \in N^\perp = M$. Thus the first map
 is well defined.
 
 Suppose moreover that $(X,\eta,G)$ is $K(N)$-general,
 i.e. $\eta(\K_X) \cap N_\RR=K(N)$. Let $\delta \in \Delta'(M)$.
 If $\P(X,\eta,G) \in H_\delta$, then $\delta \in \eta(NS(X))$,
 which is a contradiction. Thus the second map is
 well defined.
 Let $\omega \in \Omega^\chi_\Lambda \setminus \Delta$.
 By surjectivity of the period map we can find a marked
 ihs manifold  $(X,\eta)$ in $\M_\Lambda^\circ$ with
 $\P(X,\eta) = \omega$. Since $\Lambda^H$ is hyperbolic and
 $w \not \in \Delta$, we can find a K\"ahler-type chamber
 of $\eta(\NS(X))$ which contains an $H$-invariant vector
 $\kappa$. Hence we have a birational map
 $f\colon X \dashrightarrow X'$ and a monodromy operator
 $\tau \in \Mo^2(X)$ preserving the Hodge structure of
 $X$ such that $\eta'^{-1}(\kappa) \in \K_{X'}$, where
 $\eta' = \eta \circ \tau \circ f^*$.
 By construction $\eta'^{-1}H \eta'$ is a group of monodromies
 that preserves $\eta'^{-1}(\kappa)$.
 Set $G'=\rho_{X'}^{-1}(\eta'^{-1}H \eta')$.
 Then $(X',\eta',G')$ is an $H$-marked ihs manifold
 with period $w$.
 
 To see the injectivity of the second map,
 let $(X_i,\eta_i,G_i)$ for $i = 1,2$ be $K(N)$-general ihs
 manifolds with the same period. Then
 $\eta_1^{-1} \circ \eta_2$ is a monodromy operator
 compatible with the K\"ahler cones and Hodge structures.
 It is therefore induced by an isomorphism
 $f\colon X_1 \rightarrow X_2$, so that $(X_1,\eta_1)$ and
 $(X_2,\eta_2)$ are isomorphic. Since $\phi\colon \M^\chi_H \rightarrow \M_\Lambda$ is injective,
 the preimages of $(X_i,\eta_i)$ under $\phi$ are isomorphic as well.
\end{proof}

Recall that $\B\Delta(X) \leq \Delta(X) :=
\ns(X) \cap \eta^{-1}(\Delta(\Lambda))$ consist of the classes
of stably prime exceptional divisors of $X$. Let $H$ be of prime order and $\C_X^{G}= \C_X \cap H^2(X,\RR)^G$ be the connected
component of the $G$-fixed part of the positive cone which contains K\"ahler classes. We call the connected components of $\C_X^{G}\setminus \bigcup\{v^\perp : v \in \B \Delta(X)\}$ the \emph{$G$-stable exceptional chambers}.
The \emph{$G$-invariant fundamental exceptional chamber} is $\mathcal{FE}_X^G=\mathcal{FE}_X \cap H^2(X,\RR)^G$.
\begin{lemma}
If $(X,\eta,G)$ is $H$-general, then the $G$-invariant fundamental chamber is a $G$-stable exceptional chamber.
\end{lemma}
\begin{proof}
Write $\B\Delta(X)$ as the disjoint union of $\B\Delta^G(X)$,
$\B\Delta_G(X) \defeq\B\Delta(X) \cap H^2(X,\ZZ)_G$, $\B\Delta'(X)\defeq \B\Delta(X) \cap \eta^{-1}(\Delta'(M))$ and $\B\Delta''$ for the remaining $\B\Delta''\subseteq \B\Delta(X)$.
In  $\C^G_X \setminus \bigcup\{v^\perp : v \in \B \Delta(X)\}$
the hyperplanes coming from $\B\Delta''$ do not meet $\C_X^{G}$ and can be omitted. Moreover, by the generality assumption $\B\Delta_G(X)$ and $\B\Delta'(X)$ are empty.
\end{proof}

Recall that $W_{Exc}(X)\leq \Mo^2(X)$ is the subgroup generated by the reflections in stable prime exceptional divisors. We define $W^G_{Exc}(X)$ as the subgroup generated by reflections in $G$-invariant stable prime exceptional divisors.

\begin{proof}[Proof of \Cref{prop:period-map} (3)]
 Let $(X_1,\eta_1, G_1)$ and $(X_2,\eta_2, G_2)$ be in the fiber of $\P$ above $w\in \Omega^\chi_\Lambda \setminus (\Delta \cap \Delta')$.
 Since $(X_1,\eta_1)$ and $(X_2,\eta_2)$ lie in the same connected component $\mathcal{M}_\Lambda^\circ$, the composition $\psi=\eta_1^{-1} \circ \eta_2$ is a parallel transport operator, hence $\psi(\F\E_{X_2})$ is an exceptional chamber.

 Both $(X_i,\eta_i,G_i)$ are $H$-general, hence the invariant fundamental exceptional chambers $\F\E^{G_2}_{X_2}$ and $\psi^{-1}(\F\E_{X_1}^{G_1})$ are
 $G_2$-stable chambers.
 Thus there exists an element $w \in W^{G_2}_{Exc}(X_2)$ with $\psi \circ w (\F\E_{X_2}^{G_2})= \F\E_{X_1}^{G_1}$. We obtain $\psi \circ w (\F\E_{X_2})= \F\E_{X_1}$
 from the fact that two exceptional chambers are either disjoint or equal.
 By \Cref{thm:torelli} $\psi \circ w$ is induced by a birational map $f\colon X_1 \dashrightarrow X_2$.
 Recall that $G_i = \rho_{X_i}^{-1}(\eta_i^{-1}H\eta_i)$ for $i\in\{1,2\}$.
 Since $w$ is a reflection in a $G_2$-invariant class, $w \rho_{X_2}(G_2)w =\rho_{X_2}(G_2)$. Using this one obtains
$\eta_2\rho_{X_2}(fG_1f^{-1})\eta_2^{-1}
= H$,
which implies $f G_1 f^{-1}=G_2$ as desired.
\end{proof}

 Note that for $H$ a cyclic group, giving $\chi\colon H \rightarrow \CC^*$ is the same as fixing a generator $h$ of $H$ acting by a fixed primitive root of unity $\zeta$.

 \begin{theorem}\label{thm:non-symplectic-classification}
  Fix a deformation type of ihs manifolds and a connected component $\M_\Lambda^\circ$ of the moduli of $\Lambda$-marked ihs manifolds of the given deformation type.
  Let $h_1, \dots, h_n$ be a complete set of representatives of the conjugacy classes of monodromies $\Mo^2_\circ(\Lambda)$ of odd prime order $p$ with $\ker (h_i + h_i^{-1} - \zeta_p - \zeta_p^{-1})\leq \Lambda_\RR$ of signature $(2,*)$.
  Let $H_i = \langle h_i \rangle \leq \Mo^2_\circ(\Lambda)$ and $\chi_i\colon H_i \rightarrow \CC^*$ the character defined by $\chi_i(h_i)\defeq \zeta_p$.
  For each $H_i$ choose $(X_i,G_i,\eta_i) \in \M^{\circ,\chi_i}_{H_i}$.
  Then $(X_1, G_1), \dots, (X_n, G_n)$ is a complete set of representatives of pairs $(X,G)$ up to deformation and birational conjugation where $X$ is an ihs manifold of the given deformation type, $G\leq \Aut(X)$ is non-symplectic and $\ker \rho_X \leq G$, $|\rho_X(G)|=p$.
 \end{theorem}
 \begin{proof}
 Let $(X,G)$ be a pair as above. The choice of a marking $\eta$ with $(X,\eta) \in \M_\Lambda^\circ$ determines $H=\eta\rho_X(G)\eta^{-1}$ up to conjugation in $\Mo^2_\circ(\Lambda)$.
 Deformations of $(X,G,\eta)$ leave $H$ invariant.
 Similarly, birational modifications of $X$ are induced by parallel transport operators. In particular, they conjugate $H$ by a monodromy. Hence the conjugacy class of $H$ does not change under birational conjugation and deformation of $(X,G)$.
 We see that the $(X_i,G_i)$, $i=1,\dots, n$, are pairwise not equivalent.
 It remains to show that $(X,G)$ is equivalent to
 $(X_i,G_i)$ for some $i \in \{ 1, \dots, n\}$.
 Let $g\in G$ with $(g^*)^{-1} \omega_X = \zeta_p\omega_X$.
 Then $h = \eta \rho_X(g) \eta^{-1} \in H$ is independent of the choice of $g$.
 By assumption, there exists $i \in \{1,\dots n\}$ and $f \in \Mo^2_\circ(\Lambda)$ with $h = f h_i f^{-1}$.
 Then $(X,G,f\circ \eta)$ gives an element of $\M^{\circ,\chi_i}_{H_i}$.
 After a small deformation of $(X_i,G_i,\eta_i)$ and $(X,G,\eta)$, we may and will assume that they are $H$-general.
 Then each belongs to some moduli space $\M^{\circ, \chi_i}_{K_i(N)}$ respectively $\M^{\circ, \chi_i}_{K(N)}$ for suitable K\"ahler-type chambers $K_i(N)$ and $K(N)$.
 By  \Cref{prop:period-map} both moduli spaces are isomorphic to the period domain $\Omega^{\chi_i}_\Lambda \setminus (\Delta \cup \Delta')$. Since $p$ is odd, this period domain is connected. Hence after a deformation we may assume that their periods agree. We conclude with (3) of  \Cref{prop:period-map}.
\end{proof}

\section{Monodromies}
We have seen that a classification of non-symplectic automorphisms of ihs manifolds with an action of odd prime order on the second cohomology lattice
is given, up to
birational conjugation and deformation, by the corresponding conjugacy classes in
the monodromy group.
We obtain the conjugacy classes by reduction to the unimodular case which allows to handle all known deformation types
in a uniform way.

Let $(L,b)$ be a lattice
The \emph{real spinor norm} $\sigma=\mbox{spin}_\RR: O(L) \rightarrow \RR^*/\left(\RR^*\right)^2 \cong \left\{ \pm 1\right\}$ is defined as
 \[ \spin_\RR(g) = \left( -\frac{b(v_1, v_1)}{2}\right) \ldots \left( -\frac{b(v_r, v_r)}{2}\right) \mod \left(\RR^*\right)^2\]
 \noindent if $g_\RR \in O(L)_\RR$ factors as a product of reflections $g_\RR = \tau_{v_1} \circ \ldots \circ \tau_{v_r}$ with respect to elements $v_i \in (L)_\RR$ (in particular, $r \leq \rk L$ by the Cartan-Dieudonn\'e theorem). Note that  this is the real spinor norm corresponding conventionally to the quadratic form $-b$.
 \begin{example}
  Let $(L,b)$ be a lattice of signature $(l_+,l_-)$.
  We can diagonalize $b\otimes \RR$, that is, we can find
  $e_1,\dots e_n \in L \otimes \RR$ giving a diagonal Gram matrix with $l_+$ ones and $l_-$ minus ones on the diagonal.
  Then $-\id = \tau_{e_1} \circ \dots \circ \tau_{e_n}$.
  We see that the spinor norm of $-\id$ is $(-1)^{l_+}$.
 \end{example}
We set $O^+(L) = \ker \sigma$ and for $G\leq O(L)$ set $G^+=G \cap O^+(L)$.
The importance of this group lies in the fact that $\Mo^2(\Lambda) \subseteq O^+(\Lambda)$.
Here and in the following sections, root lattices associated to Dynkin diagrams are always assumed to be negative definite.

\subsection{K3 surfaces}
For any K3 surface $S$ we have $H^2(S, \IZ) \cong \Lambda = U^{\oplus 3} \oplus E_8^{\oplus 2}$, which is the unique
lattice in the genus $\even_{(3,19)}$ up to isometry. Its monodromy group is given by $\Mo^2(\Lambda)=O^+(\Lambda)$.

\begin{proof}[Proof of \Cref{thm:K3-classification}]

Since $\sigma(-\id_\Lambda)=-1$ and $-\id_\Lambda$ lies in the center of $O(\Lambda)$, conjugacy classes in $O(\Lambda)$ and $O^+(\Lambda)$ coincide.
From  \Cref{thm: classific unimod} we obtain the conditions for the existence of an isometry $f \in O(\Lambda)$ of odd prime order $p$ when $l_+ = 3$, $l_- = 19$, $s_+ = 2$. \Cref{thm:unimodular_conjugacy} implies that
the action of $f$ is determined (up to conjugacy on $\Lambda$) by the isometry class of $\Lambda^f$, which is unique by \Cref{lem:uniqueness_L^G}. We conclude with \Cref{thm:non-symplectic-classification}.
\end{proof}
\begin{remark}
 If the action of the group $G$ is symplectic, then $\Lambda_G$ is negative definite and we have the additional condition that
 $\Lambda_G$ has maximum at most $-4$. Thus our results do not apply
 to recover the classification of symplectic automorphisms, which is best understood in terms of the Niemeier lattices.
 See \cite{kondo} for a survey.
\end{remark}

\subsection{\texorpdfstring{K3$^{[n]}$}{K3 n}, \texorpdfstring{Kum$_n$}{Kum n}, \texorpdfstring{OG$_6$}{OG6} and \texorpdfstring{OG$_{10}$}{OG10}}\label{subsec: k3n}
Let $L$ be a lattice and
$\gamma: O(L) \rightarrow O(A_L)$
the natural map.
We define $\Gamma(L)=\gamma^{-1}(\{\pm \id_{A_L}\})$. Denote by $\chi = \gamma|_{\Gamma(L)}\colon\Gamma(L) \rightarrow \{\pm 1\}$
the corresponding character.

In order to treat the four deformation classes in a uniform way, we embed $\Lambda\cong H^2(X,\ZZ)$ primitively into a unimodular lattice $M$. The orthogonal complement $\Lambda^\perp \subset M$ will be denoted by $V$. These lattices are given in \Cref{tbl:invariants} (see \cite{beauville}, \cite{rapagnetta_OG6}, \cite{rapagnetta_OG10}, \cite[Lem.\ 9.2]{markman}, \cite[Cor.\ 4.8]{markman_mehrotra}, \cite[Thm.\ 4.3]{mongardi:monodromy}, \cite{mongardi_rapagnetta}, \cite{onorati}, \cite[Prop.\ 10]{beauville_rmks}, \cite[Cor.\ 5]{bnws_kum}, \cite{mongardi_wandel_OG}).

\begin{table}
\centering
\caption{Monodromy of the known deformation types of ihs manifolds.}\label{tbl:invariants}
\begin{tabular}{lllllll}
\toprule
Type& $\Lambda$ & $\mathop{\mathrm{Mon}}\nolimits^2(\Lambda)$ & $\ker \rho_X$ & $M$ & $V\leq M$ & $S$\\
\midrule
K3 & $U^{\oplus 3}\oplus E_8^{\oplus 2}$ & $ O^+(\Lambda)$ &  $1$ & $\even_{(3,19)}$& - & -\\
K3$^{[n]}$ & $U^{\oplus 3}\oplus E_8^{\oplus 2} \oplus \langle 2-2n \rangle$ & $ \Gamma^+(\Lambda)$ &  $1$ & $\even_{(4,20)}$& $\langle 2n-2 \rangle$ & $O$\\
Kum$_n$ & $U^{\oplus 3}\oplus \langle -2n-2 \rangle$ & $\ker (\det \cdot \chi)^+$ &  $C_{n+1}^4 \rtimes C_2$&$\even_{(4,4)}$&$\langle 2n+2 \rangle$ & $SO$\\
OG$_6$ & $U^{\oplus 3}\oplus \langle -2\rangle^{\oplus 2}$ & $O^+(\Lambda)$ & $C_2^8$& $\even_{(5,5)}$ & $\langle 2 \rangle\oplus \langle 2 \rangle$ & $O$\\
OG$_{10}$ & $U^{\oplus 3}\oplus E_8^{\oplus 2} \oplus A_2$ & $O^+(\Lambda)$ & $1$& $\even_{(5,21)}$&$A_2(-1)$& $O$\\
\bottomrule
\end{tabular}
\end{table}

\begin{definition}
For $\hskn$ and $\kumn$ set $h_V=-\id_V$, while for
OG$_6$ and OG$_{10}$ let
$h_V \in O(V)$ be represented by the matrix $\left(\begin{smallmatrix}
          0 & 1\\
          1 & 0
         \end{smallmatrix}\right)$ with respect to the obvious basis. Fix $v \in V$ primitive with $h_V(v)=v$.
Let $O(M,V,v)$ be the joint stabilizer in $O(M)$ of $V$ and $v$. For all deformation types the image of $O(V) \rightarrow O(A_V)$ is generated by $\bar h_V$, which has order two. This defines in a natural way a character $\chi_V\colon O(M,V,v) \rightarrow \{\pm 1\}$.
We have the restriction map $O(M,V,v) \rightarrow O(\Lambda)$ and
a section
\[\gamma \colon \Mo^2(\Lambda) \hookrightarrow O(M,V,v), \quad f \mapsto \hat{f}\defeq \hat{\chi}(f) \oplus f\]
where $\hat{\chi}(g)= \id_V$ if $\bar g = \id_{A_\Lambda}$ and $\hat{\chi}(g) = h_V$ else. In the following, let $S(M) \in \left\{O(M), SO(M) \right\}$ be as defined in \Cref{tbl:invariants} for each deformation type.

\end{definition}
\begin{lemma}
The image of $\gamma$ is given by
\[G =S^{\sigma\chi_V}(M,V,v)=\begin{cases}
O^{\sigma\chi_V}(M,V,0) & \mbox{for } \hskn,\\
SO^{\sigma\chi_V}(M,V,0) & \mbox{for } \kumn,\\
O^{\sigma\chi_V}(M,V,v) & \mbox{for } \textrm{OG}_6,\\
O^{\sigma\chi_V}(M,V,v) & \mbox{for } \textrm{OG}_{10}.
\end{cases}\]
\end{lemma}
\begin{proof}
Let $f \in \Mo^2(\Lambda)\leq O^+(\Lambda)$.
If $\bar f = \id$, then $\hat{f}= \id_V \oplus f$ and
\[(\chi_V\sigma)(\hat{f})=\sigma(\hat{f})=\sigma(f)=1.\]
Otherwise $\hat{f} = h_V \oplus f$ and
\[(\chi_V\sigma)(\hat{f})=-\sigma(\hat{f})=-\sigma(h_V)\sigma(f)=-\sigma(h_V)=1.\]

By construction $\hat{f}$ preserves $V$ and $v$.
In the $\kumn$-case we have in addition $\det(\hat{f})=\det(\hat{\chi}(f))\det(f) = \chi(f) \det(f)=1$.
We conclude that $G$ contains the image of $\gamma$.
To show the converse inclusion, let $g \in G$ and set
$f=g|_\Lambda$. 
Since $M$ is unimodular, $\chi_V(g)=1$ if and only if $\bar f = \id$.
Moreover $g\vert_V \in O(V,v)=\langle h_V \rangle$, which is of order two. Thus $\hat\chi(f) = g\vert_V$, which implies $\hat f = g$. It is immediate to check that $f \in \Mo^2(\Lambda)$.
\end{proof}
Let $A$ be a group and $g,h \in A$ elements. We denote conjugation by $\prescript{h}{}{g}=hgh^{-1}$ and conjugacy classes by $\prescript{A}{}{g}$.
The set of conjugacy classes of elements of $A$ is $\cl(A)$. For a subgroup $B \leq A$ we have a natural map
$\cl(B) \rightarrow \cl(A)$.
\begin{theorem}\label{thm:fiber}
Let
 $\psi \colon \cl(Mon^2(\Lambda))\cong \cl(G) \rightarrow \cl(S(M))$
 be the natural map.
Let $g \in S(M)$ be of odd prime order with $M_g$ of signature $(2,\ast)$ and $V \subseteq M^g$. Then the map
\[
\begin{array}{rccl}
\varphi \colon &\psi^{-1}\left(\prescript{S(M)}{}{g}\right) &\longrightarrow &S(M^g) \setminus \left\{W \subseteq M^g \text{ primitive} \mid W \cong V \right\}\\[2pt]
 &\prescript{G}{}{h} &\longmapsto&S(M^g)fV \mbox{ where } g=\prescript{f}{}{h}, f\in S(M)
\end{array}
\]
is a bijection.
\end{theorem}
\begin{proof}
 Note that $gfV=fhf^{-1}fV=fhV = fV$ and similarly $gfv=fv$. Since $p$ is odd and $\rk V \leq 2$, this implies $fV \leq M^g$. Let $s \in G$ and $f' \in S(M)$ such that
 $h' = \prescript{s}{}{h}$ and $g=\prescript{f'}{}{h'}$.
 We have to show that the left cosets $S(M^G)fV$ and $S(M^g)f'V$ coincide.
 With $t = f'sf^{-1}$ we have $g=\prescript{t}{}{g}$, thus $t$ preserves $M^g$ and its restriction is in
 $O(M^g)$. We have $tfV=f'sV=f'V$.
 For $\kumn$ we have to show in addition that $t|_{M^g} \in SO(M^g)$. Since $gt=tg$, \Cref{lem:SO} applies and gives $\det t|_{M_g}=1$. Thus $1 = \det t = \det t|_{M^g} \cdot \det t|_{M_g}= \det t|_{M^g}$.
 This shows that $\phi$ is \emph{well defined}.
 
Let $W \leq M^g$ be a primitive sublattice with $W \cong V$. By  \cite[Prop.\ 1.6.1]{nikulin} we can find $f \in O(M)$ with $W=fV$ and
since $V^\perp$ contains a hyperbolic plane, we may even assume $f \in S(M)$.
Set $h=\prescript{f^{-1}}{}{g}$. It satisfies $h|_V=\id_V$ and it is of odd order, hence $h \in G=S^{\sigma \chi_V}(M,V,v)$. By definition $g = \prescript{f}{}{h}$ and $\varphi(\prescript{G}{}{h})=S(M^g)W$. Thus $\varphi$ is \emph{surjective}.

Consider $h=\prescript{f^{-1}}{}{g}$ and $h'=\prescript{f^{-1}}{}{g}$ with
$f,f' \in S(M)$.  Their images under $\varphi$ are represented by $fV$ and $f'V$. To prove \emph{injectivity}
suppose that there is $t \in S(M^g)$ satisfying $tfV=f'V$.
As $\Gamma(V)\subseteq \im(S(M^g,V)\rightarrow O(V))$ acts transitively on the vectors of length $v^2$, we may even assume $tfv=f'v$. The natural map
 \[SO(M_g,g)=O(M_g,g) \rightarrow O(\disc{M_g})\]
 is surjective (cf.\ \Cref{prop: surjectivity}) and $M$ is unimodular, hence we can extend $t$ to an isometry $\widetilde{t} \in S(M)$ commuting with $g$. Then
 $u=(f')^{-1} \circ \widetilde{t} \circ f$ conjugates $h$ to $h'$ and preserves $V$ and $v$.
 If $(\sigma\cdot \chi_V) (u) \neq 1$ we can multiply $u$ by
 $\hat\chi(-\id_\Lambda)\oplus (-\id_\Lambda)$ to enforce $u \in G$.
\end{proof}

Together \Cref{thm: classific unimod,thm:unimodular_conjugacy,thm:non-symplectic-classification,thm:fiber}
pave the way for classifying non-symplectic monodromies of odd prime order up to conjugation.
It remains to compute the orbits of primitive sublattices in the invariant lattice $M^f$.
This is done in  \Cref{sec: primitive vectors,sec: OG10}. Before we continue with the monodromy classification,
we give the proof of  \Cref{thm: classification k3n}, which characterizes the invariant and coinvariant
lattices of non-symplectic automorphisms with an action of odd prime order on the second cohomology lattice, for ihs manifolds of type $\hskn$, $\kumn$, OG$_6$ and OG$_{10}$.

\begin{proof}[Proof of \Cref{thm: classification k3n}]
In order for $K$ to be the coinvariant lattice of a non-symplectic automorphism, its signature needs to be $(2,\rk K -2)$.
Any $p$-elementary lattice of discriminant $p$ and signature $(2,(a+2m)(p-1)-2)$ is unique
in its genus, since it is either indefinite or of signature
$(2,0)$ (and therefore isomorphic to $A_2(-1)$; see
\cite[Table 15.1]{conway_sloane}). Hence, by
\Cref{prop:clas-fixed-pt-free} the lattice $K$
admits a fixed-point free isometry $f$ of order $p$ with
$\bar f = \id_{A_K}$ if and only if it is $p$-elementary
and $\rk K = (l(A_K) + 2m)(p-1)$ for some $m \in \ZZ_{\geq 0}$. Let $T := K^\perp \subset \Lambda$. 
We may assume that $(K,f)$ has signature $(k_1^+,k_1^-)=(2,\ast)$.
By \cite[Cor.\ 1.5.2]{nikulin},
we have that $\id_T \oplus f \in O(T \oplus K)$ extends to
an isometry $\Phi \in O(\Lambda)$ such that $\Lambda^\Phi = T$ and $\Lambda_\Phi = K$. 
Since $p$ is odd this isometry is a monodromy.
\Cref{thm:non-symplectic-classification} provides $(X,\eta,\sigma)$.
\end{proof}

\begin{remark}
For manifolds of type $\hskn$, a complete classification of the pairs of invariant and coinvariant lattices of non-symplectic automorphisms of odd prime order is already known in the case $n=2$, by work of Boissi\`ere, Camere and Sarti \cite[Thm.\ 7.1]{bcs} and Boissi\`ere, Camere, Mongardi and Sarti \cite[Thm.\ 6.1]{bcms_p=23}. A similar classification has been achieved for $n=3,4$, and in all dimensions when $\rk K = 22$, by Camere and the second author \cite[Thms.\ 1.2, 4.4]{CC}. For all $n \geq 2$, an analogous statement of \Cref{thm: classification k3n} in the case of involutions has been obtained by Camere, the second author and Andrea Cattaneo in \cite[Thm.\ 2.3]{CCC} building on work of Joumaah \cite{joumaa:order2}. Moreover, for $n = 2,3,4$ and $p \neq 23$ odd, explicit examples of automorphisms realizing all possible pairs of lattices have been exhibited in \cite{bcs} and \cite{CC}.

For manifolds of type $\kumn$, Mongardi, Tari and Wandel \cite{mongardi_tari_wandel:kummer} have classified the pairs of invariant and coinvariant lattices when $n=2$ and presented examples of automorphisms for all cases where $\rk K \leq 6$.

For OG$_6$ type ihs manifolds Grossi has recently classified the invariant and coinvariant lattices \cite{grossi:nonsymplectic} of non-symplectic automorphisms of prime order.
\end{remark}

For an ihs manifold $X$ of type $\hskn$, a group of automorphisms $G \subset \aut(X)$ is called \emph{induced} if there exists a projective $K3$ surface $S$, a group $F \subset \aut(S)$, a Mukai vector $v \in H^*(S, \IZ)^F$ and a $v$-generic stability condition $\tau$ such that $(X, G)$ is isomorphic to $(M_\tau(S,v), \tilde{F})$, where $\tilde{F}$ is the induced action of $F$ on the moduli space. For type $\kumn$, the same definition holds with $S$ an abelian surface and the induced action of $F$ on the fiber $K_\tau(S,v)$ over $(0,0)$ of the Albanese map $M_\tau(S,v) \rightarrow S \times \hat{S}$. Combining \Cref{thm: classific unimod,thm:non-symplectic-classification,thm:fiber} with the results of \cite[\S 3]{mongardi_wandel:induced} we obtain the following (see also \cite[\S 2]{mongardi_tari_wandel:kummer} for $\kumn$).

\begin{proposition}
For manifolds of type $\hskn$ or $\kumn$, let $g \in O(M)$ be an isometry of odd prime order $p$ with $M^g$ of signature $(2, r-2)$ and discriminant $p^a$, for some $r = \rk M^g \geq 2$ and $a \geq 0$. Then the monodromy classes in the fiber $\psi^{-1}(\prescript{S(M)}{}g)$ admit a geometric realization as actions of induced automorphisms if and only if $M^g$ contains $U$ as a direct summand. In particular, this happens if and only if\\
$(p,r,a) \neq (3,2,1), (3,4,4), (3,6,5), (3,8,6), (5,4,3), (23,2,1)$ for type $\hskn$;\\
$(p,r,a) \neq (3,2,1), (7,2,1)$ for type $\kumn$.
\end{proposition}

Notice that, if there is a primitive embedding $U \hookrightarrow M^g$, then the fiber $\psi^{-1}(\prescript{S(M)}{}g)$ is nonempty for all values $n \geq 2$ (we can take a primitive vector of square $2(n-1)$ or $2(n+1)$ inside the copy of $U$). We also remark that, for manifolds of type $\hskn$, the isometries of $M$ with $(p,r,a) = (3,4,4), (3,6,5), (3,8,6), (5,4,3)$ can be realized as extended actions of induced automorphisms on moduli spaces of \emph{twisted} sheaves on $K3$ surface, since $M_g$ contains a primitive copy of $U(p)$ (see \cite[Thm.\ 3.4]{ckkm_twisted}). Here however, for a fixed $n$, one has to check whether there exists an embedding $V \hookrightarrow M^g$.
The complete lists of triples $(p,r,a)$ for prime order isometries of $M$ which extend non-symplectic monodromies of manifolds of type $\hskn$ and $\kumn$ are given in \Cref{fig:k3n-fix-3,fig:k3n-fix-5,fig:k3n-fix-rest,fig:kumn-fix}.

\begin{figure}
\includegraphics[scale=0.6]{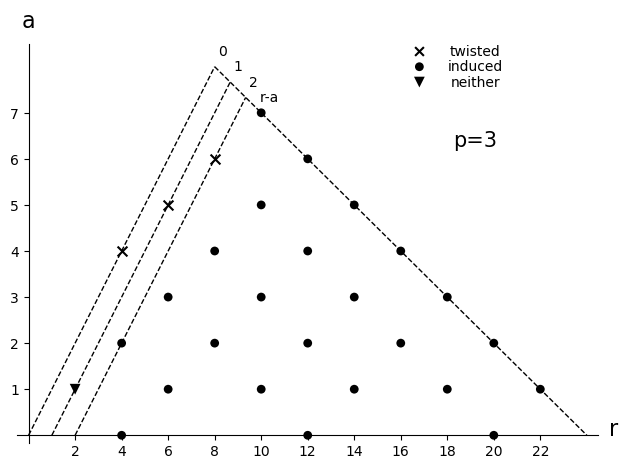}
\caption{Isometries of $\even_{(4,20)}$ of order $3$ and $\hskn$.}\label{fig:k3n-fix-3}
\end{figure}

\begin{figure}
\includegraphics[scale=0.6]{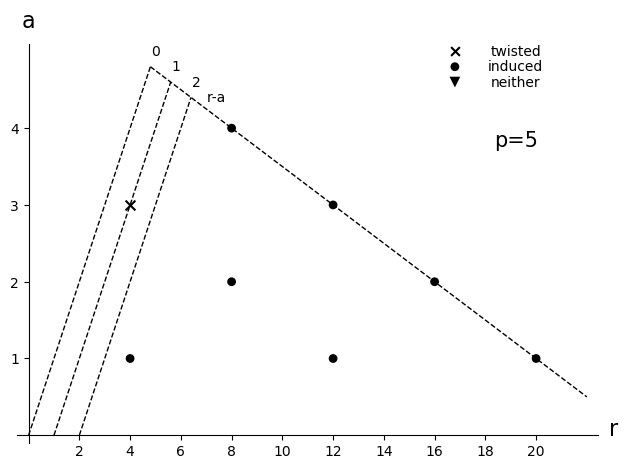}
\caption{Isometries of $\even_{(4,20)}$ of order $5$ and $\hskn$.}
 \label{fig:k3n-fix-5}
\end{figure}

\begin{figure}
\includegraphics[scale=0.6]{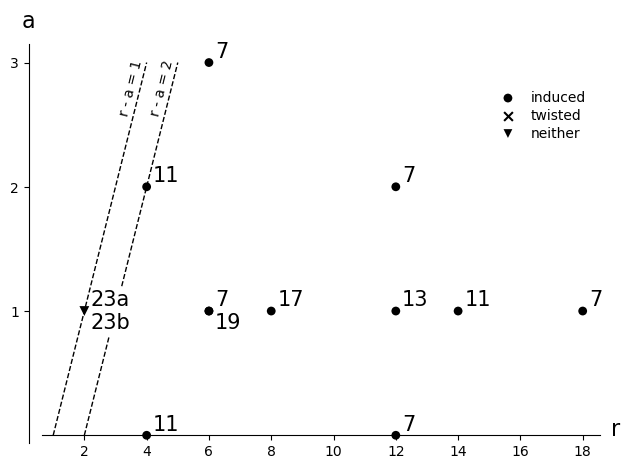}
\caption{Isometries of $\even_{(4,20)}$ of order at least $7$ and $\hskn$. Cases $23a$ and $23b$ are as in \Cref{lem:uniqueness_L^G}.}\label{fig:k3n-fix-rest}
\end{figure}

\begin{figure}
\includegraphics[scale=0.6]{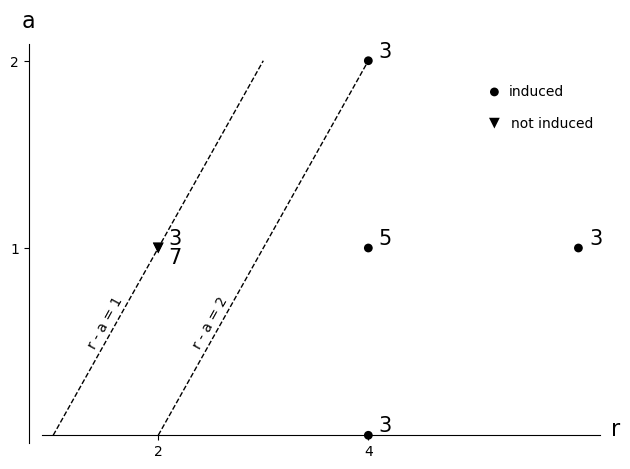}
\caption{Isometries of $\even_{(4,4)}$ and automorphisms of type $\kumn$.}
\label{fig:kumn-fix}
\end{figure}

Ihs manifolds of type OG$_6$ and OG$_{10}$ can be obtained as symplectic
desingularizations of singular moduli spaces of stable objects on
abelian surfaces and K3 surfaces respectively, hence a definition
of induced automorphisms can be given for these deformation types too
(see \cite[\S 5]{mongardi_wandel:induced}). For a lattice-theoretical
characterization of induced automorphisms of manifolds of type OG$_6$
we refer to \cite{grossi:nonsymplectic}.
 
\section{Primitive vectors in \texorpdfstring{$p$}{p}-elementary lattices and \texorpdfstring{$\hskn$, $\kumn$}{K3 n, Kum n}} \label{sec: primitive vectors}
In this section we classify orbits of primitive vectors in $p$-elementary lattices.
The classification comes in two flavors: if the lattice is positive definite, the theta series, a modular form, governs the representation behavior; if it is indefinite, then strong approximation for the spin group allows to handle orbits in terms of congruences and spinor exceptions. This boils down to linear algebra over $\FF_2$.

Let $p$ be an odd prime and $L \in \even_{(l_+,l_-)}p^{\epsilon n}$, $\epsilon \in \left\{\pm 1\right\}$, $l_+, l_- \geq 1$, $l_+ + l_- \geq 3$ and $0<k \in 2\ZZ$.
Set $R=L(-1)\oplus U$. Primitive sublattices
$\ZZ x=\langle k \rangle \subseteq L$ are determined, up the the action of $O(L)$, by
a primitive extension $\langle k \rangle \oplus R \subseteq
\overline{\langle k \rangle \oplus R}$ and
a primitive embedding of
$\overline{\langle k \rangle \oplus R}$
into the unimodular overlattice $\overline{L \oplus R}$ of $L \oplus R$ \cite[Proposition 1.15.1]{nikulin}. For later use we define the following torsion quadratic forms
\[
\begin{array}{lrcc}
u_{p,k}\colon              &(\ZZ/p^k\ZZ)^2 \longrightarrow \QQ/2\ZZ,& (x,y) \mapsto 2xy p^{-k} + 2\ZZ    ;&\\
w_{p,k}^{\epsilon}\colon   & \ZZ/p^k\ZZ    \longrightarrow \QQ/2\ZZ,& x     \mapsto  ux^2p^{-k} +2\ZZ, &u\in \ZZ, \leg{u}{p}=\epsilon.
\end{array}
\]
\subsection{Existence}
We begin by investigating the primitive extension.
\begin{lemma}\label{lem:glue-unique}
 Let $p$ be an odd prime, $\ZZ d = A\cong \ZZ/p^a\ZZ$, $a\geq 1$ and
 $B \cong (\ZZ/p\ZZ)^n$ be nondegenerate torsion quadratic
 modules.
 Suppose that there is an embedding $\phi \colon \ZZ p^{a-1}d \rightarrow B(-1)$,
 and let $H$ be its graph in $A \oplus B$.
 \begin{enumerate}
  \item Let $a=1$ and write $B=A(-1)\oplus C$. Then $H^\perp/H \cong C$.
  \item Let $a\geq 2$ and write $B = u_{p,1} \oplus C$. Then $H^\perp/H \cong A \oplus C$.
 \end{enumerate}
 In either case the map $\phi$ is unique up to the action of $O(B)$.
 Moreover $l(H^\perp/H) = l(B)-1$.
 \end{lemma}
\begin{proof}
Case (1) is clear. For the uniqueness use Witt's theorem \cite[Ch.\ 1, Thm.\ 5.3]{scharlau}.
In case (2) let $e,f \in u_{p,1}$ be generators such that
$b(e,f)=1/p$. Let $d^2 = \epsilon/p^a$. Then
$p^{a-1}d$ and its image are isotropic. By Witt's theorem, an isotropic vector in a
regular quadratic space is unique up to isometries. So we may assume that $\phi(p^{a-1}d)=e$. Then $H$ is generated by $p^{a-1}d+e$, while
$H^\perp$ is generated by $e$, $d- \epsilon f$ and $C$.
Thus $H^\perp/H \cong \ZZ/p^a \ZZ \oplus C$.
Since $(d- \epsilon f)^2=d^2$, we see that
 $\ZZ(d- \epsilon f)$ is isomorphic to $A$.
\end{proof}

We now relate the index of the primitive extension $\langle k \rangle \oplus R \subseteq \overline{\langle k \rangle \oplus R}$ to the divisibility $b(x, L)$ of the image of $x$ in $(L,b)$. A similar computation (for $L$ not necessarily $p$-elementary) can be found in \cite[Prop.\ 6.1]{garbagnati_sarti}. We denote by $\nu_p(k)$ the valuation of $k \in \ZZ$ at the prime $p$.

\begin{lemma}\label{lem:qc-decomp}
We have $\div(x)=1$ if and only if $\overline{\langle k \rangle \oplus R}=\langle k \rangle \oplus R$.
Otherwise $\div(x)=p$. Moreover
$q_{\langle k \rangle^\perp} \cong -q \oplus r$ where
\begin{enumerate}
 \item $q=q_{\langle k \rangle }$, $r = q_L$ if $\div(x) =1$;
 \item $q\oplus w_{p,1}^\epsilon=q_{\langle k \rangle }$, $r\oplus w_{p,1}^\epsilon = q_L$ if $\div(x) =p$ and $\nu_p(k) =1$;
 \item $q=q_{\langle k \rangle }$, $r \oplus u_{p,1} = q_L$ if if $\div(x) =p$ and $\nu_p(k) \geq 2$.
\end{enumerate}

\end{lemma}
\begin{proof}
 Set $K=\langle k \rangle$ and let $\pi_{\langle k \rangle}\colon L \rightarrow {\langle k \rangle}^\vee$
 be the orthogonal projection.
 Set $H_{\langle k \rangle} = \pi_{\langle k \rangle}(L)/{\langle k \rangle}$ and $h=|H_{\langle k \rangle}|= [L:{\langle k \rangle}\oplus {\langle k \rangle}^\perp]$.
 Then the divisibility of $x$ is given by the ideal
 \[\langle x, L \rangle=\langle x,\pi_{\langle k \rangle}(L) \rangle=
 \langle  x,x/h\rangle=\ZZ (k/h).\]
 
 We compute $h^2=[L:{\langle k \rangle}\oplus {\langle k \rangle}^\perp]^2 = \det {\langle k \rangle} \det {\langle k \rangle}^\perp/\det L
 $.
 Using $k =\det {\langle k \rangle}$ and $\det {\langle k \rangle}^\perp=\det \left(\overline{{\langle k \rangle}\oplus R}\right)
 = \det {\langle k \rangle} \det L /\left[\overline{({\langle k \rangle} \oplus R)}:({\langle k \rangle} \oplus R)\right]^2$
 we arrive at $k/h = \left[\overline{({\langle k \rangle} \oplus R)}:({\langle k \rangle} \oplus R)\right]$.
Since $R$ is $p$-elementary and ${\langle k \rangle}$ of rank one, this index is either $p$ or $1$.
The computation of $q_{{\langle k \rangle}^\perp}$ follows from \Cref{lem:glue-unique}.
\end{proof}

Before we compute when such a sublattice exists, we introduce some notation.
Let $s=s'p^{\nu_p(s)} \in \QQ_p$. We set $\chi_p(s)=\leg{s'}{p}$. For a torsion quadratic form $r$ we denote $\chi_p(r) = \chi_p(\det K_p(r))$ where $K_p(r)$ is the unique $p$-adic lattice of rank $l(r)$ with $q_{K_p(r)}=r$.

\begin{theorem} \label{thm: existence for p-elem}
 Let $k$ be an even positive number and $p$ be an odd prime number.
 Let $L\in \even_{(l_+,l_-)}p^{\epsilon n}$ with $l_+, l_- \geq 1$ and $l_+ + l_-\geq 3$. Write $a = \nu_p(k)$ and $k = k'p^a$.
 Then there exists a primitive vector $x \in L$
 with $x^2 = k$ and $\div(x) =1$ if and only if one of the following conditions is met:
 \begin{enumerate}
  \item[(I-1)] $a=0$, $n < \rk L - 1$;
  \item[(II-1)] $a=0$, $n = \rk L - 1$ and $\leg{-1}{p}^{l_-}\leg{k'}{p}=\epsilon$;
  \item[(III-1)] $a >0$, $n < \rk L - 2 $;
  \item[(IV-1)] $a >0$, $n = \rk L - 2$ and $l_+ - l_- \equiv 0   \pmod 8$.
 \end{enumerate}
 
 There exists a primitive vector $x \in L$
 with $x^2 = k$ and $\div(x) =p$ if and only if one of the following conditions is met:
 \begin{enumerate}
  \item[(I-p)] $a = 1$, $n > 1$;
  \item[(II-p)] $a = 1$, $n = 1$ and $\leg{k'}{p}=\epsilon$;
  \item[(III-p)] $a > 1$, $n > 2$;
  \item[(IV-p)] $a > 1$, $n = 2$ and $l_+ - l_- \equiv 0 \pmod 8$.
 \end{enumerate}

 If $l_+=1$ or $l_-=0$, the given conditions are still necessary but possibly not sufficient.
\end{theorem}
\begin{proof}
Suppose that $\div(x) =1$.
By \cite[Thm.\ 1.12.2]{nikulin} an embedding exists if and only if
$l(q \oplus -r) < \rk L - 1$  or
$l(q \oplus -r) = \rk L - 1$ and
\begin{equation}\label{eqn:nikulin}
\leg{(-1)^{l_-}|q \oplus -r|}{p} = \leg{(-1)^{l_-} k'}{p}= \chi_p(-q \oplus r).
\end{equation}

For $a=0$, $l(q\oplus -r)=\max\{1,n\}$ while for $a>0$, $l(q\oplus -r)=n+1$. This yields (I-1) and (III-1).
In case (II-1) we have $\chi_p(-q)\chi_p(r)=1\cdot \epsilon$.
In case (IV-1) we have $\chi_p(-q) = \leg{-k'}{p}$ and
$\chi_p(r)=\epsilon$.
Then \cref{eqn:nikulin}
gives $1+(p-1)(l_-+1) \equiv \leg{-1}{p}^{l_-+1}\equiv\epsilon \pmod 4$. Inserting into \cref{eqn:p-sign} we obtain $l_+ \equiv l_- \pmod 8$.

Suppose that $\div(x) = p$. For an embedding to exist, we have to be able to write
$q_C= -q \oplus r$ as in \Cref{lem:qc-decomp}.
Suppose that $a=1$. Then we need
$q_{\langle k \rangle} = q \oplus w_{p,1}^{\delta}$,
$q_L = r \oplus w_{p,1}^{\delta}$ for
$\delta = \leg{k'}{p}$. This is possible if $n> 1$ (I-p) or $n=1$ and $\epsilon=\delta$ (II-p).

Now, suppose that $a>1$. Then we can write $q=q_{\langle k \rangle}$ and $r \oplus u_{p,1} = q_L$ if and only if
$n>2$ (III-p) or $n=2$ and $q_L=u_{p,1}$, i.e.\ $\epsilon = \leg{-1}{p}\equiv 1 + (p-1) \pmod 4$. Inserting into \cref{eqn:p-sign} gives $l_+ - l_- \equiv 0$ (IV-p).

In both cases one computes $\chi_p(-q \oplus r)=\leg{k'}{p}\epsilon$.
If $\rk L - 1 = l(-q \oplus r)= n-1$, we additionally have to check  \cref{eqn:nikulin}.
By \Cref{thm:p-elementary} we have $\epsilon = \leg{-1}{p}^{l_-}$. So the condition is always satisfied.
\end{proof}

\begin{remark}
As already noted in \cite{CC}, \Cref{thm: existence for p-elem,thm:fiber} imply that the classification of non-symplectic automorphisms with a cohomology action of odd prime order $p$ on ihs manifolds of type $\hskn$ (resp.\ $\kumn$) is usually richer when $p$ divides $2(n-1)$ (resp.\ $2(n+1)$), since in this case the Mukai vector inside the invariant lattice of the extended action on the Mukai lattice might be chosen of divisibility $p$ other than $1$.
\end{remark}

\subsection{Spinor exceptions}\label{subs: spinor}
At this point we can effectively determine when a primitive sublattice $\langle k \rangle \hookrightarrow L$ exists.
Further, we know that $O(L)$-orbits of primitive sublattices
are locally determined by their divisibility. However, globally
this need not be true, due to spinor exceptions.
In this subsection the exceptional cases will be quantified.

\textbf{Setting.}
Let $p$ be an odd prime and $L \in \even_{(l_+,l_-)}p^{\epsilon n}$, $\epsilon \in \left\{\pm 1\right\}$, $l_+ \geq 2$, $l_-\geq 1$, and $l_+ + l_- \geq 4$. Let $0<k \in 2\ZZ$ and $\ZZ x=\langle k \rangle \subseteq L$ be a primitive sublattice and set $C = \langle k \rangle^\perp$.
We will write $q_C = -q \oplus r$ as in \Cref{lem:qc-decomp}. We set $p^m = |r|$ and have $j=|q|=j' p^b$ with $b=\nu_p(j)$.
Note that our conditions on $l_\pm$ imply that $C$ is indefinite of rank at least $3$.

\textbf{Notation.}\hfill
In order to compute the spinor exceptions,
we rely on the results of Miranda and Morrison
\cite{miranda_morrison:embeddings}.
We recall the necessary notation.
\begin{itemize}
 \item $|\mathfrak{g}(C)|$ the number of isomorphism classes of lattices in the genus of $C$,
 \item $\AA$ the finite adele ring of $\QQ$,
 \item $S$ the finite set of square free integers dividing $2\det C$,
  \item $\UU_\ell= \QQ_\ell^\times / \left(\QQ_\ell^\times\right)^2$
 \item $\EE_\ell= \{s \in \UU_\ell \mid \nu_\ell(s) \equiv 0 \pmod 2\}$,
 \item $\Gamma_{\ell}=\{ \pm 1 \} \times \UU_\ell$,
 \item $\Gamma_{\ell,0}=\{\pm 1\} \times \EE_\ell$,
 \item $\Gamma_\AA=\left\{(d_\ell,s_\ell)_{\ell \in \PP} \in \prod_{\ell \in \PP}\Gamma_{\ell} \mid (d_\ell,s_\ell) \in \Gamma_{\ell,0} \mbox{ for all but finitely many  } \ell\right\}$,
 \item $\Gamma_\QQ= \{\pm 1\} \times \QQ^\times / \left(\QQ^\times\right)^2\leq \Gamma_\AA$,
 \item $\Gamma_S=\{(d,s) \in \Gamma_\QQ \mid s \in S\}$,
 \item $O(L\otimes \mathbb{A})=\left\{f \in \prod_{\ell \in \PP} O(L_\ell\otimes \QQ_\ell) \mid f_\ell \in O(L_\ell) \mbox{ for all but finitely many }\ell\right\}$,
 \item $\spin(\tau_x)\defeq b(x,x)/2$ for $x \in L$ anisotropic and $\tau_x$ the reflection in $x$,
 \item $\sigma = (\det,\spin)\colon O(L\otimes \mathbb{A})
 \rightarrow \Gamma_\AA$,
 \item $\Sigma(L_\ell) = \im(\sigma_\ell \colon O(L_\ell) \rightarrow \Gamma_\ell)$,
 \item $\Sigma^\#(L_p) = \im(\sigma_\ell \colon O^\#(L_\ell) \rightarrow \Gamma_\ell)$,
 \item $\Sigma(L) = \prod_{\ell \in \PP}\Sigma(L_\ell) \leq \Gamma_\AA$,
 \item $\Sigma^\#(L) = \prod_{\ell \in \PP}\Sigma^\#(L_\ell)$.
\end{itemize}
Note that we view $\Gamma_\ell$ as a subgroup of $\Gamma_\AA$.

\begin{lemma}[Key Lemma]
 Write $q_C = -q \oplus r$ as in \Cref{lem:qc-decomp} and set
 \[\Delta(C)\defeq\left(\Gamma_S \cap \Sigma(C)\right)\cdot \Sigma^\#(C)
 \cdot \sigma\left(O(r)\times \{\pm \id_{-q}\}\right). \]
 Then the number of primitive sublattices $\langle k \rangle \subseteq L$ with $q_{\langle k \rangle^\perp} \cong q_C$ up to the action of $O(L)$ is given by $|\mathfrak{g}(C)|\cdot [\Sigma(C):\Delta(C)]$.
 \end{lemma}
\begin{proof}
 Fix (the isomorphism class of) the orthogonal complement $C$ and let $\phi$ be the obvious anti-isometry
 \[\phi \colon q_{\overline{\langle{k}\rangle \oplus R}}= q \oplus -r  \longrightarrow -q \oplus r=q_C.\]
 
 Let $H= \overline{\langle k \rangle \oplus R}/(\langle k \rangle \oplus R)$ be its graph and $T = \mbox{Stab}(H,O(\langle k\rangle) \times O(R))$ the corresponding stabilizer.
 Then, following \cite[Prop.\ 1.15.1]{nikulin}, the number of primitive sublattices $\langle k \rangle$ with orthogonal
 complement isomorphic to $C$ is given by the double coset
 \[
  \overline{O(C)}\backslash O(q_C)/ \phi
  \overline{T} \phi^{-1}
 \]
 where the bar indicates that we regard the image in the orthogonal group of the discriminant group.
 Since $\overline{O(C)}$ is normal, we can turn this into the single coset
 \begin{equation}\label{eqn:spin_sequence}
O(q_C) \; /\; \left(\overline{O(C)} \cdot
 \phi \overline{T} \phi^{-1}
 \right).
 \end{equation}

Now $\overline{O(C)}$ is computed by the exact sequence \cite[VIII Thm.\ 5.1]{miranda_morrison:embeddings}
 \[O(C) \rightarrow O(q_C) \xrightarrow{\sigma}
\Sigma(L)/(\Gamma_S\cap \Sigma(C))\cdot
\Sigma^\#(C)\rightarrow 1\]
 where by abuse of notation we identify $\sigma$ and the corresponding map in the sequence.
 We see that
 \[ O(q_C)/ (\overline{O(C)} \cdot \phi
  \overline{T} \phi^{-1}) \cong
\Sigma(L)/(\Gamma_S\cap \Sigma(C)\cdot
\Sigma^\#(C))\cdot \sigma(\phi\overline{T}\phi^{-1}).\]

It remains to compute $\phi\overline{T}\phi^{-1}$.
We pass to the discriminant group $A_{\langle k \rangle} \oplus A_R$. Since $O(R)\rightarrow O(q_R)$ is surjective,
we may replace $O(R)$ by $O(q_R)$ and $O(\langle k \rangle)$ by $\{\pm \id_{q_{\langle k\rangle }}\}$. Looking at the three cases in \Cref{lem:qc-decomp}, we see that $\phi \overline{T} \phi^{-1}= O(r) \times \{\pm \id_{-q}\}$.
\end{proof}

We continue by computing the pieces making up $\Sigma(C)$ and $\Delta(C)$.
\begin{lemma}
 Let $\ell$ be a prime, $\ell \neq 2, p$. 
 \begin{equation}\label{eqn:SCq}
\Sigma^\#(C_\ell)= \Gamma_{\ell,0},
 \qquad
\Sigma(C_\ell)=
 \begin{cases}
 \Gamma_{\ell,0} & \mbox{for }\nu_\ell(j) \equiv 0 \pmod 2\\
  \Gamma_\ell &\mbox{for } \nu_\ell(j) \equiv 1 \pmod 2
 \end{cases}
 \end{equation}
 \begin{equation}\label{eqn:SC2}
\Sigma^\#(C_2)= \Gamma_{2,0},
 \qquad
\Sigma(C_2)=
 \begin{cases}
 \Gamma_{2,0} & \mbox{for }\nu_2(j) \equiv 1 \pmod 2\\
  \Gamma_2 &\mbox{for } \nu_2(j) \equiv 0 \pmod 2
 \end{cases}
 \end{equation}
 \end{lemma}
\begin{proof}
This follows from \cite[VII \S 12]{miranda_morrison:embeddings} with $\rk C \geq 3$ and $l(A_C)_\ell \leq 1$.
\end{proof}

\begin{lemma}\label{lem:SCp-sharp}
Let $q_C \cong -q \oplus r$ with $q,r$ as in \Cref{lem:qc-decomp}.
 Let $\delta = \det C / \det (q_C)_p$.
  \begin{equation}\label{eqn:SCp_sharp}
\Sigma^\#(C_p)=
\begin{cases}
\langle (1,1)\rangle & \mbox{for } \quad l(A_C) = \rk C\\
\langle (-1,2\delta )\rangle& \mbox{for }\quad  l(A_C) + 1 = \rk C\\
\Gamma_{p,0}& \mbox{for } \quad l(A_C) + 2 \leq \rk C.
\end{cases}
 \end{equation}
 \end{lemma}
\begin{proof}
This is a translation of
\cite[VII Thm.\ 12.1]{miranda_morrison:embeddings}
to our notation.
\end{proof}
Recall that $|r|=p^m$ and $|q|=j= j'p^b$.
\begin{lemma}\label{lem:SCp}
Let $q_C \cong -q \oplus r$ as in \Cref{lem:qc-decomp}. Let $j=j' p^b$, $b=\nu_p(j)$.
Let $0 \neq u \in \EE_p$.
 \[\Sigma(C_p) = \begin{cases}
    \Gamma_{p,0} &\mbox{for }\quad m=0 \wedge b \equiv_2 0,\\
  \langle (1,u),(-1,p)\rangle &\mbox{for }\quad l(A_C)=\rk C \wedge(b=0 \vee  b \equiv_2 1),\\
  \langle (-1,1),(-1,p)\rangle  &\mbox{for }\quad m=1 \wedge b>1 \wedge \rk C=3 \wedge \left( \frac{-j'}{p}\right)=\left(\frac{2}{p}\right) \hypertarget{ast}{\condAst},\\
  \Gamma_p &\mbox{else.}
    \end{cases}\]
\end{lemma}
\begin{proof}
This is a translation of \cite[VII Thms.\ 12.5, 12.7, 12.9]{miranda_morrison:embeddings} to our situation.
The only nontrivial computation is ($\ast$).

Assume that
$m=1$, $b>1$ and $\rk C = 3$.
Then the rank of $L$ is $4$. By our assumptions the signature of $L$ is either $(2,2)$ or $(3,1)$.
Let $1^{\epsilon_0} p^{\epsilon_1} (p^b)^{\epsilon_b}$
be the $p$-adic genus symbol of $C_p$.
Suppose that $\div(x)=1$ so that $|\det L| = p$.
By the classification of $p$-elementary genera in \Cref{thm:p-elementary},
$L$ must be in the genus $\even_{(2,2)}p^{\epsilon_1}$
(respectively $\even_{(3,1)}p^{\epsilon_1}$)
for $p \equiv 1 \pmod 4$ (respectively $p \equiv 3 \pmod 4$).
Let $\div(x)=p$. Then
$L \in \even_{(2,2)}p^{\epsilon_1' 3}$ (respectively $\even_{(3,1)}p^{\epsilon_1' 3}$) for $p \equiv 1 \pmod 4$ (respectively $p \equiv 3 \pmod 4$). By \Cref{lem:qc-decomp}, $\epsilon_1 = \leg{-1}{p} \epsilon_1'$.

With \Cref{thm:p-elementary} we compute that $\epsilon_1 = \left(\tfrac{2}{p}\right)$.
We have $\det C = (-1)^{l_-}j' p^{b+m}= \left(\tfrac{-1}{p}\right) j' p^{b+m}$ and
$\epsilon_b = \left(\tfrac{-j'}{p}\right)$.
By comparing the determinant with the symbol, we obtain the relation
\begin{equation}\label{eqn:eps}
\left(\frac{-j'}{p}\right)=\chi_p(\det C)=\epsilon_0 \epsilon_1 \epsilon_b=\epsilon_0 \epsilon_1 \left(\frac{-j'}{p}\right).
\end{equation}

Hence $\epsilon \defeq \epsilon_0 = \epsilon_1$.
The conditions of \cite[VII Thm.\ 12.5]{miranda_morrison:embeddings} hold if and only if $\epsilon_b=\epsilon$ giving $\condAst$.
It remains to compute the spinor norms with \cite[VII Thms.\ 12.5, 12.9]{miranda_morrison:embeddings}.
Let $\alpha \in \EE_p$ be defined by $\left(\tfrac{\alpha}{p}\right)=\epsilon=\left(\tfrac{2}{p}\right)$.
Then $\Sigma(C_p) =\langle (-1,2\alpha),(-1,2 \alpha p)\rangle$. We conclude with $\left(\tfrac{2 \alpha}{p}\right)=1$.
\end{proof}
\begin{lemma}\label{lem:sigma-rq}
Write $q_C = -q \oplus r$ as in \Cref{lem:qc-decomp}.
Generators of
\[
 \sigma\left(O(r)\times \{\id_{-q_j}\}\right) \leq \Sigma(C)/\Sigma^\#(C) \]
 are given by
 \begin{enumerate}
  \item $(-1,2\alpha p)_p$ for $m=1$ and $r \cong w_{p,1}^\epsilon$ where $\epsilon = \leg{\alpha}{p}$,
  \item $(-1,p)_p, (1,u )_p$ for $m \geq 2$ and $1 \neq u \in \EE_p$.
 \end{enumerate}
\end{lemma}
\begin{proof}
 We know that $O(r)$ is generated by reflections.
 So one may just compute the spinor norms of all reflections.
 For $m\geq 2$ use that $w_{p,1}^{1} \oplus w_{p,1}^{1} \cong w_{p,1}^{-1} \oplus w_{p,1}^{-1}$.
\end{proof}

\begin{lemma}
 \[\Delta(C)=\left(\Gamma_S \cap \Sigma(C)\right)\cdot \Sigma^\#(C)
 \cdot \Sigma(C_p) \]
\end{lemma}
\begin{proof}
Combine \Cref{lem:SCp-sharp,lem:SCp,lem:sigma-rq} to obtain
\begin{equation}\label{eqn:sigma_Cp}
\Sigma(C_p)=
\Sigma^\#(C_p)\cdot \sigma_p\left(O(q_L)\times \{ \id_{-q}\}\right) \cdot \langle (1,t)_p \rangle
\end{equation}
where $t=p$ if $l(A_C)=\rk C$, $0<b\equiv 0 \pmod 2$ and else $t=1$.
If $t\neq 1$, then $\Sigma(C_p)=\Gamma_p$ and
$\Sigma^\#(C_p)= \langle (1,1)\rangle$.
We rewrite
\begin{eqnarray*}
\Delta(C) \cdot \langle (1,t)_p\rangle &=& \left(\Gamma_S \cap \Sigma(C)\right)\cdot
\Sigma^\#(C) \cdot \sigma\left(O(r)\times \{ \id_{-q}\}\right) \cdot \sigma(\{\pm \id_{q_C}\})\cdot \langle (1,t)_p\rangle\\
&=& \left(\Gamma_S \cap \Sigma(C)\right)\cdot
\Sigma^\#(C) \cdot \sigma\left(O(r)\times \{ \id_{-q}\}\right) \cdot\langle (1,t)_p\rangle\\
 &=& \left(\Gamma_S \cap \Sigma(C)\right)\cdot
\Sigma^\#(C) \cdot \Sigma(C_p)
\end{eqnarray*}
where we use that $-\id_{q_C} \in \overline{O(C)}$, i.e.\ $\sigma(-\id_{q_C}) \in (\Gamma_S \cap \Sigma(C))\cdot \Sigma^\#(C)$ and \cref{eqn:sigma_Cp}.

It remains to show that $(1,t)_p \in \Delta(C)$.
So suppose that $t=p$. Then $(1,p) \in \Gamma_S \cap \Sigma(C)$ and $(1,p)_\ell \in \Gamma_{\ell,0}= \Sigma^\#(C_\ell)$ for $\ell\neq p$. Thus $(1,p)_p \in (\Gamma_S\cap\Sigma(C))\cdot \Sigma^\#(C)\subseteq \Delta(C)$.
\end{proof}

\begin{definition}
Define
\[\mathcal{L}_1 = \{\ell \in \PP\setminus \{2,p\} \mid \nu_\ell(j) \equiv_2 1\} \cup \{2 \mid 0 <\nu_2(j) \equiv_2 0 \}\]
\[\mathcal{L}_0 = \{\ell \in \PP\setminus \{2,p\} \mid 0< \nu_\ell(j)  \equiv_2 0\} \cup \{2 \mid \nu_2(j) \equiv_2 1\}\]
so that the set of prime divisors of $pj$ is the disjoint union $\{p\} \sqcup \mathcal{L}_1 \sqcup \mathcal{L}_0$.
For $\ell \in \mathcal{L}_1$ we have  $\Sigma(C_\ell)=\Gamma_{\ell}$.
For $\ell \in \mathcal{L}_0$ we have $\Sigma(C_\ell)=\Gamma_{\ell,0}.$
 \end{definition}

\begin{proposition}
The index of $[\Sigma(C) : \Delta(C)]$ is one, unless $\condAst$ is true and $\mathcal{L}_1$
 contains a nonsquare modulo $p$. In this case the index is two.
\end{proposition}
\begin{proof}
The set $\mathcal{L}_1$ is defined so that for $\ell \in \mathcal{L}_1$ the quotient $\Sigma(C_\ell)/\Sigma^\#(C_\ell) = \Gamma_{\ell}/\Gamma_{\ell,0}$ is of order $2$ and generated by the class of $(1,\ell)$.
The map
\[\phi\colon \Sigma(C)/\left(\Sigma^\#(C)\cdot \Sigma(C_p)\right) \longrightarrow \FF_2^{\mathcal{L}_1}, \quad(d_\ell,s_\ell)_{\ell \in \PP} \mapsto (\ell \mapsto \nu_\ell(s_\ell))\]
is an isomorphism of $\FF_2$-vector spaces.
Hence the index is given by \[[\Sigma(C) : \Delta(C)]=[\FF_2^{\mathcal{L}_1} : \phi(\Gamma_S \cap \Sigma(C))].\]

Let $(d,s) \in \Gamma_S$. It lies in $\Sigma(C)$ if and only if $(d_p,s_p) \in \Sigma(C_p)$.
Suppose that we are not in case $\condAst$. Then $\{1\} \times \EE_p \leq \Sigma(C_p)$. Hence $(1,\ell) \in \Gamma_S \cap \Sigma(C)$ for all $\ell \in {\mathcal{L}_1}$. Taking the images $\phi(1,\ell)$, we obtain the standard basis of $\FF_2^{\mathcal{L}_1}$.

If $\condAst$ holds, $\Sigma(C_p)=\langle (-1,1),(1,p)\rangle$.
Then $(d,s)$ with $s=s'p^{\nu_p(s)}$ is in $\Gamma_S \cap \Sigma(C)$
if and only if $\left(\frac{s'}{p}\right)=1$.
This gives a linear condition on $\FF_2^{{\mathcal{L}_1}}$ which is
nontrivial if and only if there is at least one $s \in {\mathcal{L}_1}$ with $\left(\frac{s}{p}\right)=-1$.
\end{proof}
Notice the isomorphism
$\Gamma_p \rightarrow \FF_2^3$, $(d,s) \mapsto (\epsilon_p(d),u_p(s),\nu_p(s))$, where $d=(-1)^{\epsilon_p(d)}$ and $u_p = \epsilon_p\circ \chi_p$.
\begin{lemma}
 We have $|\mathfrak{g}(C)|=1$, unless $\condAst$ is true,
 $p \equiv 1 \pmod 4$ and
 every element of ${\mathcal{L}_1}$ is a square modulo $p$,
 in which case $|\mathfrak{g}(C)|=2$.
\end{lemma}
\begin{proof}
The order of $\mathfrak{g}(C)$ is the order of the cokernel of
\[\Gamma_S \rightarrow
\Delta' / \{(\pm 1, \pm 1)\}_S \]
where $\Delta'\defeq \left(\prod_{\ell \mid 2jp}  \Gamma_\ell / \Sigma(C_\ell)\right)$ (cf. \cite[VIII Prop. 6.1]{miranda_morrison:embeddings}).
For $\ell\in \mathcal{L}_0$ we have $\Gamma_\ell/\Sigma(C_\ell)=\Gamma_\ell/\Gamma_{\ell,0}$ and for $\ell \in \mathcal{L}_1$ the quotient is trivial.
Let $\pi_p\colon \Gamma_p / \Sigma(C_p) \xrightarrow{\sim} (\FF_2)^\gamma$, $\gamma \in \{0,1\}$ be the unique isomorphism.
We obtain the isomorphism
\[
\psi \colon \Delta' = \Gamma_p/\Sigma(C_p) \times
\prod_{\ell \in \mathcal{L}_0}  \Gamma_\ell / \Gamma_{\ell,0}
\xrightarrow{\sim} \FF_2^\gamma \times \FF_2^{\mathcal{L}_0}, \quad (d,s) \mapsto  (\pi_p(d,s),\ell\mapsto \nu_\ell(s)).\]

A basis $B$ of $\Gamma_S$ is given by
\[\{(1,\ell) \mid \ell \in {\mathcal{L}_1} \cup \mathcal{L}_0\}\cup \{(1,p)\} \cup \{(\pm 1,\mp 1)\}.\]
We view $\psi(B)$ as a matrix and compute its rank.
The projection of $\{\psi(1,\ell) \mid \ell\in \mathcal{L}_0\}$ to $\FF_2^{\mathcal{L}_0}$ is a basis of $\FF_2^{\mathcal{L}_0}$. Thus $\psi(B)$ has rank at least $|\mathcal{L}_0|$. It has rank $|\mathcal{L}_0| + \gamma$ if and only if $\gamma=0$ or $\{\pi_p(1,s) \mid s \in {\mathcal{L}_1} \cup \{p\}\}\cup \{\pi_p(\pm 1,\mp1)\}$ contains a nonzero element.
If $\condAst$ does not hold and $\gamma \neq 0$, then $\pi_p \in \{\nu_p,\epsilon_p\cdot \nu_p\}$. Now $\pi_p(1,p)\neq 0$ does the trick.
If $\condAst$ holds, then $\pi_p = u_p$.
For $p \equiv 3 \pmod 4$ we have $u_p(1,-1)\neq 0$.
But for $p \equiv 1 \pmod 4$ we have $u_p(\pm 1,\mp 1)= 0 =u_p(1,p)$.
Thus the map is not surjective if and only if $p\equiv 1 \pmod 4$ and  every $\ell \in {\mathcal{L}_1}$ is a square mod $p$.
\end{proof}
\begin{example}
 For $k=50$, $L = H_5 \oplus U \in \even_{(2,2)}5^{-1}$ and divisibility $1$ we have $j=2$, $b=2$, $m=1$, $\mathcal{L}_0=\{2\}$, $\mathcal{L}_1=\emptyset$.
 Thus the genus of $C$ consists of the two classes
 \[  \left(\begin{array}{rrr}
2 & 3 & 0 \\
3 & -8 & 0 \\
0 & 0 & -10
\end{array}\right), \left(\begin{array}{rrr}
2 & 1 & 0 \\
1 & -2 & 10 \\
0 & 10 & -90
\end{array}\right).\]

For a case where $|\mathfrak{g}(C)|=1$ but $[\Sigma(C) : \Delta(C)]=2$, consider
$L = H_5 \oplus U$ and divisibility $1$, so that $j = k$. We need $\leg{-j'}{5}=\leg{2}{5}=-1$ and that $\mathcal{L}_1$ contains a nonsquare modulo $p$. The smallest example is $k= 4 \cdot 3 \cdot 25 = 300$.
\end{example}
We summarize the results of this section in the following theorem.
\begin{theorem}\label{thm:orbits}
 Let $p$ be an odd prime number, $k= k'p^{\nu_p(k)}>0$ even and $L$ a $p$-elementary lattice of signature $(l_+,l_-)$ with $l_+\geq 2$, $l_-\geq 1$ and $l_+ + l_- \geq 4$.
 The number of $O(L)$-orbits of primitive vectors $x \in L$ of a given divisibility with $x^2=k$ is either zero or one, unless
 \[\rk L = 4,\quad \nu_p(k)\geq 2,\quad |\det L| /\div(x)^2=p,\quad \left(\frac{-2k'}{p}\right)=1\]
 and either $p \equiv 1 \pmod 4$, or $p\equiv 3 \pmod 4$ and there is $\ell\in {\mathcal{L}_1}$ with $\leg{\ell}{p}=-1$.
 In this case the number of orbits is two.
 \end{theorem}
\begin{remark}
 The number of $O(L)$-orbits of rank one primitive sublattices and primitive vectors is the same since $-\id_L$ exchanges the two generators of the sublattice. For $SO(L)$ this is certainly the case if $\det(-\id_L)=1$, i.e. the rank of $L$ is even.
\end{remark}

\begin{lemma}
 Let $L$ be as in \Cref{thm:orbits} and $V \subset L$ a primitive sublattice of rank one. Then $SO(L)V = O(L)V$.
\end{lemma}
\begin{proof}
We have $SO(L)V = O(L)V$ if and only if $O(L) = SO(L)$ or $\exists f \in O(L)$, $\det f=-1$ and $fV=V$.
If the rank of $L$ is odd, then $f = -\id_L$ suffices. If $\rk L$ is even, after replacing $f$ by $-f$ we can assume that $f\vert_V = \id_V$.
Such $f$ exists if and only if there is $g \in O(C)$ with $g|_{-q}=\id_{-q}$ (and then $g=f|_C$).
By \Cref{lem:SCp-sharp,lem:sigma-rq} we find
$(-1,s) \in \sigma(\{\id_{-q}\} \times O(r)) \cdot \Sigma^\#(C)$ for some $s \in \QQ$. Inspecting the proof of \cite[VIII Thm.\ 2.3]{miranda_morrison:embeddings} shows that
 we can find $g\in O(C)$ with $\sigma(g) = (-1,s)$.
\end{proof}

\subsection{The definite case} \label{appendix:fixed}
Let $L$ be an even positive definite lattice of even dimension $n$.
Its \emph{level} is defined as the smallest natural number $N$  with $N L^\vee \subseteq L$.
 Its \emph{theta series} is
 \[\Theta_L(z) = \sum_{x \in L }q^{x^2/2}= \sum_{k=0}^\infty a(k)q^k, \quad \mbox{with } q=\exp(2\pi i z).\]
 It is a modular form of weight $n/2$ with respect to the congruence subgroup $\Gamma_0(N)$ of SL$_2(\ZZ)$ and the character $\chi(d)=\leg{(-1)^{n/2}\det L}{d}$ (see e.g.\ \cite[Thm.\ 3.2]{ebeling}).
 
 The integer $a(k) = \#\{x \in L \mid x^2 = 2k\}$ counts the number of vectors of length $2k$ in $L$.
 Define $r(k) = \#\{x \in L \text{ primitive} \mid x^2 = 2k\}$
 as the number of \emph{primitive} vectors of length $2k$.
 We can compute $r$ from $a$ as follows.
 Since every vector is a multiple of a primitive vector,
 we can write
 \begin{equation}\label{eqn:a}
 a(k) = \sum_{d^2 \mid k} r(k/d^2).
 \end{equation}
 
 Consider the space $\CC^\NN$ of complex sequences equipped with point wise addition and the Dirichlet convolution $\ast$. It is a commutative ring with identity element $\epsilon=(1,0,0,\dots)$.
 Denote by $\mu$ the M\"obius function, by $\lambda$ the Liouville function, by $1=1_\NN$ the constant sequence $(1,1,1,\dots)$ and by $1_{sq}$ the indicator function of the set of square numbers.  We note that $1_{sq}=\lambda \ast 1$, $\mu \ast 1 = \epsilon$ (M\"obius inversion) and $|\mu| \ast \lambda = \epsilon$.
 In this terminology \cref{eqn:a} is $a = 1_{sq} \ast r = (\lambda \ast 1) \ast r$.
 Thus $r = |\mu| \ast \mu \ast a$. This yields
 \begin{equation}\label{eqn:primitive}
r(n)= \sum_{d^2 \mid n}\mu(d)a(n/d^2)
 \end{equation}
where we have used that $(\mu \ast |\mu|)(d)=0$ if $d$ is not a square and $(\mu \ast |\mu|)(d)=\mu(\sqrt{d})$ if $d$ is a square.

In our classification of invariant lattices of prime order isometries, there are four positive definite invariant lattices of rank $2$:
\[A_2(-1)=\left(\begin{matrix}2 & 1 \\1 & 2\end{matrix}\right),\quad
K_7=\left(\begin{matrix}2 & 1 \\1 & 4\end{matrix}\right),\quad
F_{23a}=\left(\begin{matrix}2 & 1 \\1 &12\end{matrix}\right),\quad
F_{23b}=\left(\begin{matrix}4 & 1 \\1 & 6\end{matrix}\right).\]

Set $\theta_2\defeq \sum_{m = -\infty}^{\infty} q^{(m+\tfrac{1}{2})^2}$ and
$\theta_3\defeq \sum_{m = -\infty}^{\infty} q^{m^2}$.
Following \cite[Ch.\ 4, \S 6.2]{conway_sloane} the first three lattices have theta series given by $\theta_3(z)\theta_3(ez) + \theta_2(z)\theta_2(ez)$ for $e=3,7,23$. For the last one $F_{23b}$ we have $\theta_3(z)\theta_3(ez) + \theta_2(z)\theta_2(ez) - 2\eta(z)\eta(23z)$ (see also \cite{blij}). To see this note that the difference $\Theta_{23a}-\Theta_{23b}$ is a cusp form of weight $1$ and level $23$. The space of such forms is of dimension one \cite[\href{http://www.lmfdb.org/ModularForm/GL2/Q/holomorphic/23/1/b/a/}{Modular Form 23.1.b.a}]{lmfdb}. It is spanned by $\eta(z)\eta(23z)$ where the Dedekind $\eta$-function is $\eta=q^{1/24}\prod_{n=1}^\infty(1-q^n)$.

Denote by $b_k$ the number of primitive vectors of length $2k$ in $L$ modulo the action of $G \in \{SO(L),O(L)\}$.
A primitive vector is stabilized by $g \in G$ if and only if it is an eigenvector of eigenvalue $1$ for $g$. Since $L$ is of rank $2$ this eigenspace is of rank at most $1$ if $g\neq 1$.
Thus there are only finitely many primitive $v \in L$ with nontrivial stabilizer in $G$.
This means that for $k \notin S_G \defeq \{v^2/2 \mid \ZZ v = \ker (1-g), g\in G, g\neq 1\}$ we have $b_k = r(k)/|G|$. For the readers convenience we list the first few terms in \Cref{tbl: orbit lengths}.

\begin{table}
\centering
\caption{Orbit lengths of primitive vectors.} \label{tbl: orbit lengths}
\begin{tabular}{lll}
\toprule
$L$ & $G$ & $\sum b(k) t^k$\\
\midrule
$A_2(-1)$ &  $O$ &$t + t^{3} + t^{7} + t^{13} + t^{19} + t^{21} + t^{31} + t^{37} + t^{39} + t^{43} + t^{49}\dots$\\
$A_2(-1)$ & $SO$ &$t + t^{3} + 2t^{7} + 2t^{13} + 2t^{19} + 2t^{21} + 2t^{31} + 2t^{37} + 2t^{39} + 2t^{43}  \dots$\\
$K_7$ & $SO$ &$t + 2t^{2} + 2t^{4} + t^{7} + 2t^{8} + 2t^{11} + 2t^{14} + 2t^{16} + 4t^{22} + 2t^{23} \dots$\\
$F_{23a}$ & $O$ &$t + t^{6} + t^{8} + t^{12} + t^{18} + t^{23} + t^{26} + t^{27} + t^{36} + t^{39} + t^{48}\dots$\\
$F_{23b}$ & $O$ &$t^{2} + t^{3} + t^{4} + t^{6} + t^{9} + t^{12} + t^{13} + t^{16} + t^{18} + 2t^{24} + t^{26} + t^{29}\dots$\\
\bottomrule
\end{tabular}
\end{table}

\begin{remark}
 We note that $r(k)$ and thus $b(k)$ is unbounded.
\end{remark}

For $K \leq L$ a primitive sublattice we define its divisibility $\div(K)=|\det K| / [L : K \oplus K^\perp]$.

\begin{definition}
 Let $(X,G)$ be a pair consisting of an ihs manifold $X$ of type $\hskn$ or $\kumn$ and $G \leq \Aut(X)$ a group of non-symplectic automorphisms with $\ker \rho_X \leq G$, $\rho_X(G)$ of odd prime order $p$. Fix a generator $f \in \rho_X(G)$ acting by $\zeta_p$ on the holomorphic $2$-form $\omega_X$. Let $(p,r,a)$ be the invariants of the extended action of $f$ on the unimodular lattice $M \supset H^2(X,\IZ)$ as in \Cref{subsec: k3n} and $[I_f]$ the Steinitz class of the $\ZZ[\zeta_p]$-module $(H^2(X,\ZZ)_f,f|_{H^2(X,\ZZ)_f})$

 We say that $G$ is \emph{ambiguous} if $(X,G)$ is not determined up to birational conjugation and deformation by $(p,r,a)$, $[I_f]$ and
 the divisibility of $V \leq M$.
 \end{definition}

\begin{remark}
 If $p<23$, then $[I_f]$ is trivial (\Cref{rmk: relative class numbers}).
 The invariants $(p,r,a)$ and $\div(V)$ are equivalent
 to the datum of the genera of the invariant and coinvariant lattices $H^2(X,\ZZ)^f$ and $H^2(X,\ZZ)_f$.
\end{remark}
 \begin{table} 
 \caption{Ambiguous automorphisms for $\hskn$.}\label{tbl:ambiguous-k3n}
  \begin{tabular}{lll}
   \toprule
   $(p,r,a)$ & $\div(V)$ & $n$\\
   \midrule
   $(3,2,1)$   & $-$ & $92, 134, 218, 248, 260, 274, 302, 400, 404, 428, 470, 482, \dots$\\
   $(5,4,1)$   & $1$ & $26, 101, 126, 151, 226, 276, 351, 401, 476, 501, 526, 601, \dots $\\
   $(5,4,3)$   & $5$ & $26, 101, 126, 151, 226, 276, 351, 401, 476, 501, 526, 601, \dots$\\
   $(23,2,1)$ & $-$ & $ 7, 13, 19, 25, 27, 37, 40, 49, 53, 55, 59, 63, 73, 79, 83,88 \dots$ \\
   \bottomrule
  \end{tabular}
\end{table}

 \begin{table}
 \caption{Ambiguous automorphisms for $\kumn$.}\label{tbl:ambiguous-kumn}
  \begin{tabular}{lll}
   \toprule
   $(p,r,a)$ & $\div(V)$ & $n$\\
   \midrule
   $(3,2,1)$   & $-$ & $6, 12, 18, 20, 30, 36, 38, 42, 48, 56, 60, 66, 72, 78, 90, 92,96,\dots $\\
   $(5,4,1)$   & $1$ & $24, 99, 124, 149, 224, 274, 349, 399, 474, 499, 524, 599, 624, \dots$\\
   $(7,2,1)$ & $-$ & $3, 7, 10, 13, 15, 21, 22, 27, 28, 31, 36, 42, 43, 45, 52, 55, 57, 63 \dots$\\
\bottomrule
  \end{tabular}
\end{table}

\begin{corollary}
If a group of automorphisms of an ihs manifold of type $\hskn$ or $\kumn$ is ambiguous, then its invariants are as in \Cref{tbl:ambiguous-k3n,tbl:ambiguous-kumn}.
\end{corollary}

\begin{remark}
Differently from the case of K3 surfaces, for ihs manifolds of higher dimension it is possible to have non-symplectic automorphisms with birationally conjugate deformations, but whose geometrical fixed loci are topologically different. An example for manifolds of type K3$^{[2]}$ is discussed in \cite[Rmk.\ 7.7]{bcs} and \cite[§7.2]{bcs:ball}. However, for a manifold $X$ of type K3$^{[2]}$ with an automorphism $g$ of odd prime order $p \leq 19$ the numerical invariants of the coinvariant lattice determine the dimension of the mod $p$ cohomology $h^*(X^g, \mathbb{F}_p)$ and the Euler characteristic $\chi(X^g)$, and vice-versa (see \cite[Thm.\ 2]{smith} and \cite[\S 3.2, Appendix A]{bcs}).
\end{remark}

\section{Embeddings of \texorpdfstring{$A_2(-1)$}{A2(-1)} and \texorpdfstring{OG$_{10}$}{OG10}}\label{sec: OG10}
In this section we compute the number of orbits of primitive sublattices isomorphic to $A_2(-1)$ into the invariant lattices of prime order isometries of $\even_{(5,21)}$.

\begin{proposition}
 Let $L \in \even_{(3,l_-)}p^{\epsilon n}$ with $l_- > 0$, $l_- \equiv 1 \pmod 2$.
 Then $A_2(-1)$ embeds primitively into $L$ with divisibility $1$ if and only if one of the following holds
 \begin{enumerate}
  \item[(I-1)] $p\neq 3$, $\rk L - n > 2$;
  \item[(II-1)] $p\neq 3$, $\rk L - n = 2$ and $\epsilon = \leg{-3}{p}$;
  \item[(III-1)] $p =   3$, $\rk L - n > 3$;
  \item[(IV-1)] $p =   3$, $\rk L - n = 3$ and $\epsilon = -1$.
 \end{enumerate}
The lattice $A_2(-1)$ embeds primitively into $L$ with divisibility $3$ if and only if $p=3$ and one of the following holds
 \begin{enumerate}
  \item[(I-3)] $n=1$ and $\epsilon=-1$;
  \item[(II-3)] $n>1$, $\rk L - n > 1$;
  \item[(III-3)] $n>1$, $\rk L - n = 1$ and $\epsilon = 1$.
 \end{enumerate}
 The embedding of $A_2(-1)$ is determined up to the action of $O(L)$ by its divisibility.
\end{proposition}
\begin{proof}
 Recall the notation used in \Cref{subs: spinor}. We use \Cref{lem:qc-decomp} with $q_{\langle k \rangle}$ replaced by $q_V$.
 Note that $A_2(-1)$ embeds into $L$ if and only if
 $l(q_C)< \rk L -2$ or $l(q_C)=\rk L-2$ and
 \[\chi_p((-1)^{l_-} 3 p^n) =\chi_p(-3)= \chi_p(q_C).\]
 
 If $p = 3$, this yields $-1=\chi_3(q_C)$.
  Suppose that the divisibility is $1$. Then
  $q_C=q_{A_2} \oplus q_L=w_{3,1}^1 \oplus q_L$.
 If $p \neq 3$, then $l(q_C)=n$ and $\chi_p(q_C)=\chi_p(q_L)=\epsilon$ giving (I-1) and (II-1).
 If $p = 3$, then $l(q_C)=n+1$ and $\chi_3(q_C)=\epsilon$ giving (III-1) and (IV-1).
 Let $\div(V)=3$, then $p=3$ and we can write $q_C= r$ where $q_L = w_{3,1}^{-1} \oplus r$ if and only if $n>1$ or $n=1$ and $q_L=w_{3,1}^{-1}$, i.e.\ $\epsilon=-1$.
 Then $l(q_C)=n -1$ and $\chi_3(q_C)=-\epsilon$ giving (I-3), (II-3), (III-3).

 The analogue of $\condAst$ in this situation never holds since $b \leq 1$. So we have uniqueness of the embedding for $l_- >1$. If $l_-=1$, then $\rk C=2$ so the previous reasoning does not apply. However there are just three cases to consider. We leave them to the reader.
\end{proof}

\begin{corollary}
 Let $X$ be an ihs manifold of type OG$_{10}$ and $G \leq \aut(X)$ a group of odd prime order $p$ of non-symplectic automorphisms. Let $f \in \rho_X(G)$ be a generator acting by $\zeta_p$ on the holomorphic $2$-form $\omega_X$. Then $(X, G)$ is determined up to deformation and birational conjugation by the isomorphism classes of $H^2(X,\ZZ)^f$, $H^2(X,\ZZ)_f$ and additionally, if $p=23$, by the Steinitz class $[I_f]$ of the $\ZZ[\zeta_p]$-module $(H^2(X,\ZZ)_f,f|_{H^2(X,\ZZ)_f})$.
\end{corollary}

\Cref{fig:OG10-fix-3,fig:OG10-fix-5,fig:OG10-fix-rest} provide the triples $(p, r, a)$ for isometries $f$ of $M \in \even_{(5,21)}$ of odd prime order $p$, with $\rk M^f = r$, $\det M^f = p^a$ and $M_f$ of signature $(2, \ast)$. We highlight the cases where $M^f$ contains a primitive copy of $V = A_2(-1)$, and its divisibility.

\begin{figure}
\includegraphics[scale=0.6]{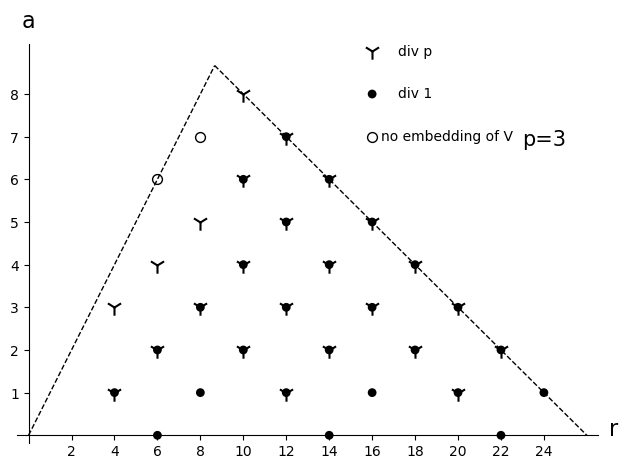}
\caption{Isometries of $\even_{(5,21)}$ of order $3$ and OG$_{10}$.} \label{fig:OG10-fix-3}
\end{figure}

\begin{figure}
\includegraphics[scale=0.6]{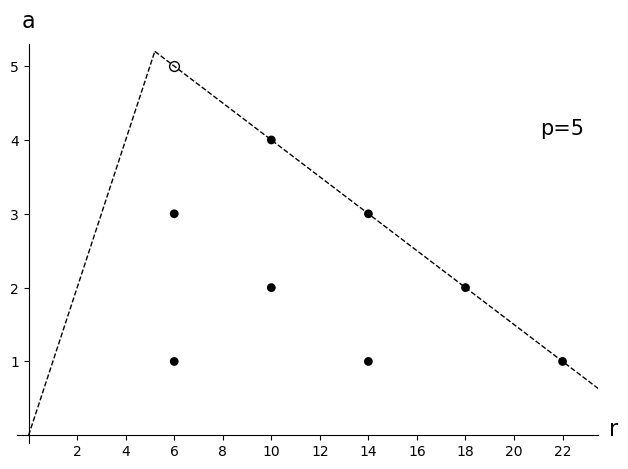}
\caption{Isometries of $\even_{(5,21)}$ of order $5$ and OG$_{10}$.}
\label{fig:OG10-fix-5}
\end{figure}

\begin{figure}
\includegraphics[scale=0.6]{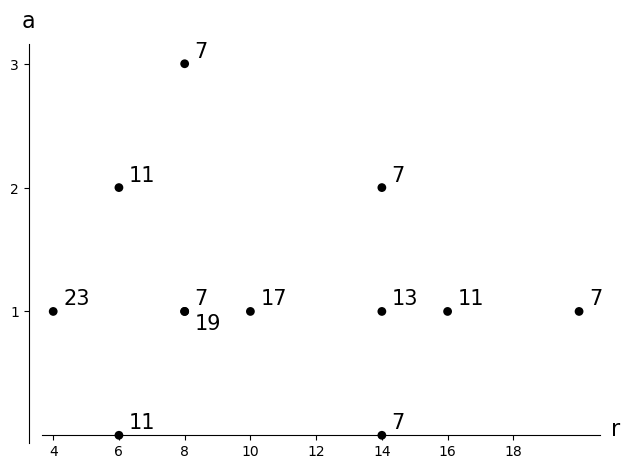}
\caption{Isometries of $\even_{(5,21)}$ of order at least $7$ and OG$_{10}$.}
\label{fig:OG10-fix-rest}
\end{figure}
\bibliographystyle{alpha}
\bibliography{literature}

\end{document}